\documentclass[11pt,reqno,oneside]{amsart}
\usepackage{newtxtext,newtxmath}
\usepackage{amsfonts,amsmath,mathtools,enumerate,dsfont}

\usepackage{xypic}
\usepackage[a4paper]{geometry}
\usepackage{hyperref}
\usepackage{cleveref}
\usepackage[dvipsnames]{xcolor}
\usepackage{tikz-cd}

\hypersetup{colorlinks=true,citecolor=NavyBlue,linkcolor=Brown,urlcolor=Orange}

\usepackage{paralist, xspace, graphicx, url, amscd, euscript, verbatim, mathrsfs,stmaryrd,epic,eepic}
\usepackage[indent]{parskip}

\numberwithin{equation}{section}

\newtheorem{theorem}{Theorem}[section]
\newtheorem{proposition}[theorem]{Proposition}
\newtheorem{conjecture}[theorem]{Conjecture}
\newtheorem{corollary}[theorem]{Corollary}
\newtheorem{lemma}[theorem]{Lemma}

\theoremstyle{definition}
\newtheorem{definition}[theorem]{Definition}
\newtheorem{question}[theorem]{Question}
\newtheorem{example}[theorem]{Example}
\newtheorem{remark}[theorem]{Remark}

\DeclareMathOperator{\Aut}{Aut}
\DeclareMathOperator{\ad}{ad}
\DeclareMathOperator{\Spec}{Spec}
\DeclareMathOperator{\Br}{Br}

\DeclareMathOperator{\disc}{disc}
\DeclareMathOperator{\End}{End}
\DeclareMathOperator{\ess}{ess}
\DeclareMathOperator{\hdg}{Hdg}
\DeclareMathOperator{\id}{id}
\DeclareMathOperator{\MT}{MT}

\DeclareMathOperator{\rk}{rk}
\DeclareMathOperator{\ord}{ord}
\DeclareMathOperator{\Frac}{Frac}

\newcommand{\Cl}{{\rm Cl}}
\newcommand{\Pic}{{\rm Pic}}
\newcommand{\KS}{{\rm KS}}
\newcommand{\NS}{{\rm NS}}
\newcommand{\Gal}{{\rm Gal}}
\newcommand{\GL}{{\rm GL}}
\newcommand{\SO}{{\rm SO}}

\newcommand{\GSpin}{{\rm GSpin}}
\newcommand{\Nef}{{\rm Nef}}
\newcommand{\Amp}{{\rm Amp}}
\newcommand{\BA}{{\rm BA}}
\newcommand{\Bir}{{\rm Bir}}
\newcommand{\Mov}{{\rm Mov}}
\DeclareMathOperator{\et}{\Acute{e}t}

\newcommand{\GSp}{{\rm GSp}}
\newcommand{\Pos}{{\rm Pos}}
\newcommand{\Sh}{\mathrm{Sh}}
\renewcommand{\hom}{{\rm hom}}
\newcommand{\Sha}{\mathrm{Shaf}}

\newcommand{\scF}{\mathscr{F}} 
\newcommand{\scS}{\mathscr{S}}

\newcommand\rH{\mathrm{H}}
\newcommand{\rT}{\mathrm{T}}
\newcommand\rP{\mathrm{P}}
\newcommand\rI{\mathrm{I}}
\newcommand\spin{{\bf sp}}

\newcommand{\NN}{\mathbb{N}}
\newcommand{\CC}{\mathbb{C}}
\newcommand{\QQ}{\mathbb{Q}}
\newcommand{\RR}{\mathbb{R}}
\newcommand{\ZZ}{\mathbb{Z}}
\newcommand{\fX}{\mathfrak{X}}
\newcommand{\PP}{\mathbb{P}}

\newcommand{\p}{\mathfrak{p}}

\DeclareMathOperator{\Sp}{sp}
\DeclareMathOperator{\bir}{-bir}

\DeclarePairedDelimiterX\Set[1]\lbrace\rbrace{\def\given{\;\delimsize\vert\;}#1}

\newif\ifHideFoot
\HideFoottrue  %
\HideFootfalse  %

\ifHideFoot

\newcommand{\Lie}[1]{}
\newcommand{\Zhiyuan}[1]{}
\newcommand{\Haitao}[1]{}
\newcommand{\Teppei}[1]{}

\else 

\newcommand{\marg}[1]{\normalsize{{
			\color{red}\footnote{{\color{blue}#1}}}{\marginpar[\vskip
			-.25cm{\color{red}\hfill$\Rightarrow$\tiny\thefootnote}]{\vskip
				-.2cm{\color{red}$\Leftarrow$\tiny\thefootnote}}}}}
\newcommand{\Lie}[1]{\marg{(Lie) #1}}
\newcommand{\Zhiyuan}[1]{\marg{(Zhiyuan) #1}}
\newcommand{\Haitao}[1]{\marg{(Haitao) #1}}
\newcommand{\Teppei}[1]{\marg{(Teppei) #1}}

\fi

\AtBeginDocument{%
	\def\MR#1{}
}

\title{Unpolarized Shafarevich conjectures for hyper-K{\"a}hler varieties}

\author[L. Fu]{Lie Fu}
\address{Institut de recherche math\'ematique avanc\'ee (IRMA), Universit\'e de Strasbourg, France}
\email{lie.fu@math.unistra.fr}

\author[Z. Li]{Zhiyuan Li}\address{Shanghai Center for Mathematical Sciences, Fudan University, 2005 Songhu Road 200438, Shanghai, China}\email{zhiyuan\_li@fudan.edu.cn}

\author[T. Takamatsu]{Teppei Takamatsu}
\address{Department of Mathematics, Faculty of Science, Saitama University, 255 Shimo-Okubo, Sakura-ku, Saitama-shi, Saitama 338-8570, Japan}
\email{teppeitakamatsu.math@gmail.com}

\author[H. Zou]{Haitao Zou}
\address{Universität Bielefeld, Universitätsstraße 25, 33615, Bielefeld
Germany}
\email{hzou@math.uni-bielefeld.de}

\subjclass[2020]{14G35, 14J28, 14J42, 11G15, 11G35}
\keywords{K3 surfaces, Hyper-Kähler varieties, Shafarevich conjectures, Shimura varieties, uniform Kuga--Satake construction, arithmetic period map, CM type}
\begin{document}

\thanks{L.~F. is supported by the University of Strasbourg Institute for Advanced Study (USIAS), and by the Agence Nationale de la Recherche (ANR) under projects ANR-20-CE40-0023 and ANR-24-CE40-4098.  Z.L. is supported by NSFC grants (12121001, 12171090 and 12425105) and the Shanghai Pilot Program for Basic Research (No. 21TQ00).  H.~Z. is  supported by  Deutsche Forschungsgemeinschaft (DFG, German Research Foundation) -- SFB-TRR 358/1 2023 -- 491392403, NKRD Program of China (No. 2020YFA0713200) and LMNS. T.~T. is supported by JSPS KAKENHI Grant number JP19J22795 and JP22KJ1780.}

\begin{abstract} 
The Shafarevich conjecture/question is about the finiteness of isomorphism classes of a family of varieties defined over a number field with good reduction outside a finite collection of places. For K3 surfaces, such a finiteness result was proved by Y.~She. For hyper-Kähler varieties, which are higher-dimensional analogues of K3 surfaces, Andr\'e proved the Shafarevich conjecture for hyper-Kähler varieties of a given dimension and admitting a very ample polarization of bounded degree. In this paper, we provide a unification of both results by proving the (unpolarized) Shafarevich conjecture for hyper-Kähler varieties in a given deformation type. We also discuss the cohomological generalization of the Shafarevich conjecture by replacing the good reduction condition with the unramifiedness of the cohomology, where our results are subject to a certain necessary assumption on the faithfulness of the action of the automorphism group on cohomology.
In a similar fashion, generalizing a result of Orr and Skorobogatov on K3 surfaces, we prove the finiteness of geometric isomorphism classes of hyper-Kähler varieties of CM type in a given deformation type defined over a number field with bounded degree. A key to our approach to these results is a uniform Kuga--Satake map, inspired by She's work, and we study its arithmetic properties, which are of independent interest.
\end{abstract}
\maketitle
\setcounter{tocdepth}{1}
\tableofcontents
\section{Introduction}
\subsection{Background}
Let $K$ be a number field and $S$ a finite set of places of $K$. The classical Hermite--Minkowski theorem says that there are only finitely many extensions of $K$ with a fixed degree that is unramified outside $S$. The geometric generalization is the so-called \textit{Shafarevich question}: given a family $\mathcal{M}$ of smooth projective varieties defined over $K$, we ask the following:
\begin{question}[Finiteness of varieties] \label{Question}
Is the set $\Sha_{\mathcal{M}}(K, S)$ of $K$-isomorphism classes of varieties in $\mathcal{M}$ defined over $K$ and with good reduction outside $S$ finite? 
\end{question}
As the notation indicates, such question can be traced back to a famous conjecture of Shafarevich \cite{ShafarevichICM} for the family of smooth curves of a given genus ($\geq 2$), which plays an important role in Faltings' proof of the Mordell conjecture in \cite{Faltings83}.  Such Shafarevich-type questions have been investigated in various situations, where results fall into two categories: polarized versus unpolarized.

For {\it polarized varieties}, i.e.~varieties equipped with an ample line bundle, the question has an affirmative answer in a number of cases:
\begin{itemize}
	\item (Faltings \cite{Faltings83}) Curves of a fixed genus $g\geq 2$, i.e.~$\Sha_{\mathcal{M}_g}(K,S)$ is a finite set. 
	\item (Faltings \cite{Faltings83}) Abelian varieties of a fixed dimension $g$ and admitting a polarization of a given degree $d$, i.e.~$\Sha_{\mathcal{A}_{g,d}}(K,S)$ is a finite set. 
	\item  (Scholl \cite{Sch85}) del Pezzo surfaces.
	\item (Andr\'e \cite{A96}) K3 surfaces admitting a polarization of a given degree. 
	\item (Andr\'e \cite{A96}) Hyper-Kähler varieties (with $b_2>3$) of given dimension with a very ample polarization of bounded degree; and similarly for more general ``K3-type" varieties such as cubic fourfolds. 
	\item (Javanpeykar--Loughran \cite{Javanpeykar-FlagVarieties}) Flag varieties.
	\item (Javanpeykar \cite{Javanpeykar-CanPolarizedSurface}) Canonically polarized surfaces fibered over a curve. 
	\item (Javanpeykar--Loughran \cite{JavanpeykarCI2017}, \cite{JavanpeykarLoughran-SexticSurface}) Certain varieties ``controlled'' by abelian varieties (via e.g.~intermediate Jacobian): complete intersections of Hodge level $\leq 1$, prime Fano threefolds of index 2, sextic surfaces, etc.
    \item (Lawrence and Sawin \cite{LawrenceSawin})  Hypersurfaces,  up to translation, in a given abelian scheme of dimension $\geq 4$ over $\mathcal{O}_{K,S}$, representing a given ample class. This is based on the techniques in Lawrence--Venkatesh \cite{LawrenceVenkatesh}.
\end{itemize}
These results can often be reinterpreted as the finiteness of $\mathcal{O}_{K,S}$-points in certain moduli spaces; see works of Javanpeykar and his coauthors \cite{JavanpeykarLoughran-StackyCW21}, \cite{Javanpeykar-MathAnn21}, \cite{Javanpeykar-CyclicCover}, \cite{Javanpeykar-PointedFamily} for related studies from this point of view of arithmetic hyperbolicity. 

The \textit{unpolarized} Shafarevich conjecture is a much stronger statement that bypasses the restriction on polarizations (e.g.~degree or natural embedding) in the finiteness statements in the above polarized version. The first example is Zarhin's result \cite{ZarhinTrick85} which gives a positive answer to Question \ref{Question} for abelian varieties of a given dimension $g$, i.e.~the following set is finite:
\[\bigcup_{d}\Sha_{\mathcal{A}_{g,d}}(K, S),\]
generalizing the aforementioned theorem of Faltings.
As for K3 surfaces, Y. She \cite{Sh17} established  the unpolarized Shafarevich conjecture, strengthening Andr\'e's result:
\begin{theorem}[She]\label{Thm:She}
The following set is finite:
	\begin{equation*}
		\Sha_{\rm K3}(K,S)=\Set*{X \given \parbox{13em}{$X$ is a K3 surface over $K$,\\
		having good reduction outside $S$}}/ \cong_{K}.
	\end{equation*}
\end{theorem}
The analogue of Zarhin's trick for K3 surfaces also has its origin in Charles' proof of the Tate conjecture \cite{Ch16}, as well as in Orr--Skorobogatov \cite{OrrSkorobogatov}.

More recently, the third author \cite{Takamatsu2020K3} proved an even stronger version of Theorem \ref{Thm:She} by weakening the good reduction condition to a condition on unramified places of the second cohomology as Galois modules, hence establishing the so-called \emph{cohomological (unpolarized) Shafarevich conjecture} for K3 surfaces.
He also established the analogous results for bielliptic surfaces \cite{takamatsu-Bielliptic}, Enriques surfaces \cite{Takamatsu2021a}, and proper hyperbolic polycurves \cite{Nagamachi2019} (the last one, jointly with Nagamachi, generalizes \cite{Javanpeykar-CanPolarizedSurface}).

\subsection{Finiteness of hyper-Kähler varieties with bounded bad reduction}
In this paper, we aim to generalize She's result (Theorem \ref{Thm:She}) to higher-dimensional analogues of K3 surfaces, which are hyper-Kähler varieties (see Section \ref{sec:HKgenerality} for generalities on this type of varieties). We propose the following conjecture. As before, we fix a number field $K$ and a finite set of places $S$. A smooth projective $K$-variety is called hyper-Kähler if the associated complex variety is hyper-Kähler (for an/any embedding of $K$ into $\mathbb{C}$); see Section \ref{sec:HKgenerality}.

\begin{conjecture}[{Unpolarized Shafarevich conjecture for hyper-Kähler varieties}]\label{conj1}
 Given a positive integer $n$, there are only finitely many $K$-isomorphism classes of  $2n$-dimensional hyper-Kähler varieties defined over $K$ and having good reduction outside $S$.
\end{conjecture}

In this generality, Conjecture \ref{conj1} seems out of reach, considering the topological difficulty that even the finiteness of deformation classes of complex hyper-Kähler manifolds in a given dimension is unknown. It is therefore natural to restrict ourselves to a given \textit{deformation type}. Let $M$ be a deformation type of complex hyper-Kähler manifolds (e.g.~$M$ can be $K3^{[n]}$, $\mathrm{Kum}_n$, $OG_6$, or $OG_{10}$). We say a hyper-Kähler variety $X$ over $K$ is of deformation type $M$, if $X\times_{\Spec K} \Spec(\mathbb{C})$ is of deformation type $M$, for some embedding $K\hookrightarrow \mathbb{C}$; see Section \ref{subsec:DeformationType}.

The following \textit{Shafarevich set} is the central object of study in this paper:
	\begin{equation}\label{HKset}
	\Sha_{M}(K,S)=\Set*{ X \given \parbox{20em}{$X$ is a hyper-Kähler variety over $K$ of deformation type $M$, having good reduction outside $S$}
    }/ \cong_{K}.
	\end{equation}

Generalizing the cohomological Shafarevich conjecture for K3 surfaces proposed by the third author in \cite{Takamatsu2020K3}, let us also consider the following \textit{cohomological Shafarevich set}, which is larger than \eqref{HKset}:
\begin{equation}\label{homHKset}
	\Sha_{M}^{\hom}(K,S)=\Set*{ X \given \parbox{20em}{$X$ is a hyper-Kähler variety over $K$ of deformation type $M$, with $\rH^2_{\et}(X_{\bar{K}},\QQ_{\ell})$ unramified outside $S$}
    }/ \cong_{K}.
	\end{equation}
Here, we say a $\Gal(\overline{K}/K)$-module is unramified at a place $v$ if the action of the inertia group $\rI_{\bar{v}}$ is trivial, where $\bar{v}$ is any extension to $\overline{K}$ of the valuation $v$.

Compared to the case of K3 surfaces, a new feature of higher-dimensional hyper-Kähler varieties that causes substantial difficulties is the existence of non-isomorphic birational transformations among them. Birational isomorphisms between hyper-Kähler varieties are isomorphisms in codimension one and hence preserve (see Proposition \ref{prop:bir-obs}) the cohomological Shafarevich condition in \eqref{homHKset} on the unramifiedness of the second cohomology. In particular, $\Sha_{M}^{\hom}(K,S)$ contains the entire $K$-birational class of any of its members. This motivates us to consider the so-called \textit{Shafarevich set with essentially good reduction}, denoted by $\Sha_{M}^{\ess}(K,S)$; see Definition \ref{essentially good} for the precise definition. Roughly speaking, a hyper-Kähler variety $X$ defined over $K$ has essentially good reduction at a place $v$ if, after a finite extension $K'/K$ that is unramified at $v$, we have a $K'$-birational model of $X_{K'}$ that has good reduction.

The above three Shafarevich sets are related as follows:
\begin{equation}
    \Sha_{M}(K,S) \subset \Sha_{M}^{\ess}(K,S)\subset \Sha_{M}^{\hom}(K,S).
\end{equation}

\vspace{0.5cm}

The first main result of the paper is the following, which gives a positive answer to Conjecture \ref{conj1} for all hyper-Kähler varieties, \textit{upon fixing the deformation type}. The statement is in a more general setting: we allow finitely generated field extension $F$ of $\QQ$ instead of just a number field $K$, the role of the ring of $S$-integers in $K$ is played by a finitely generated normal domain $R$ with fraction field $F$, and the good reduction condition is stipulated only for prime ideals of height 1. See Section \ref{sec:ShafarevichSets} for the precise definitions of various Shafarevich sets. In the general notation, the previous $\Sha_M(K,S)$ corresponds to $\Sha_M(K, \mathcal{O}_{K,S})$.

\begin{theorem}[{Unpolarized Shafarevich conjecture in a deformation class}]\label{mainthm}
Let $R$ be a finitely generated $\ZZ$-algebra which is a normal integral domain with fraction field $F$. Let $M$ be a deformation type of hyper-Kähler varieties with $b_2\neq 3$.
Then there are only finitely many $F$-birational isomorphism classes in  $\Sha^{\ess}_M(F,R)$.\\ 
If $b_2\geq 5$, then $\Sha^{\ess}_M(F,R)$, and hence $\Sha_M(F,R)$, is a finite set. 
\end{theorem}

The finiteness of deformation classes of hyper-Kähler manifolds in a given dimension is unknown. Nevertheless, Huybrechts \cite{Hu03} proved such a finiteness result upon fixing the Beauville--Bogomolov form. This has been strengthened by Kamenova \cite{Kamenova2018a} in the following form: 

\begin{theorem}[Huybrechts, Kamenova] \label{finitetype}
There are only finitely many deformation classes of complex hyper-Kähler manifolds, with given Fujiki constant and given discriminant of the Beauville--Bogomolov form. 
\end{theorem}

The definitions of the Fujiki constant and the discriminant, which are natural deformation invariants for hyper-Kähler varieties, are recalled in Section \ref{sec:HKgenerality}.
Combining Theorem \ref{finitetype} with Theorem \ref{mainthm}, we obtain the following result in the direction of Conjecture \ref{conj1}:
\begin{theorem}\label{mainthm2}
Let $K$ be a number field and $S$ a finite set of places of $K$.
For any $n, \Delta\in \NN$ and $c\in \QQ$, there are only finitely many $K$-isomorphism classes of $2n$-dimensional hyper-Kähler varieties defined over $K$, with Fujiki constant $c$, discriminant of Beauville--Bogomolov form $\Delta$, $b_2\geq 5$, and with good reduction outside $S$. 
\end{theorem}
 The gap between Theorem \ref{mainthm2} and Conjecture \ref{conj1} can be resolved if there are only finitely many possibilities for Fujiki constants and discriminants of the Beauville--Bogomolov forms of hyper-Kähler manifolds in a given dimension.  

\vspace{0.5cm}
Another new feature of hyper-Kähler varieties in dimension $>2$ is that the action of the automorphism group on the second cohomology is not necessarily faithful (see Example \ref{exm:aut}), which by \textit{twisting} \cite{BorelSerre} leads to the failure of the naive version of the cohomological generalization of the Shafarevich conjecture for (polarized or unpolarized) hyper-Kähler varieties in general, see Section \ref{subsec:Fail-secondCohomShafConj} for counter-examples. On the other hand, we have the following result, where $(i)$ says that passing to geometric isomorphism classes (hence ignoring the effect of twists) recovers the finiteness, and $(ii)$ says that the nonfaithfulness of the action of the automorphism group is the ``only" obstruction to finiteness.

\begin{theorem}[{Cohomological Shafarevich conjecture}]
\label{mainthm-cohom}
Let $R$ be a finitely generated $\ZZ$-algebra that is a normal integral domain with fraction field $F$. Let $M$ be a deformation type of hyper-Kähler varieties with $b_2\neq 3$.
\begin{enumerate}
    \item [(i)] There are only finitely many geometrically birational isomorphism classes in  $\Sha_M^{\hom }(F,R)$.\\ If $b_2\geq 5$, then there are only finitely many geometric isomorphism classes in $\Sha_M^{\hom}(F,R)$. 
     \item [(ii)] Suppose that for any hyper-Kähler variety in the deformation type $M$, the automorphism group acts faithfully on the second cohomology. Then there are only finitely many $F$-birational isomorphism classes in  $\Sha_M^{\hom}(F,R)$.\\ If $b_2\geq 5$,  then $\Sha_M^{\hom}(F,R)$ is a finite set.
\end{enumerate}
\end{theorem}  

In fact, to remedy the nonfaithfulness of the action of the automorphism group on the second cohomology, we investigate in Section \ref{subsec:FullDegCohomShafConj} a more general version of the cohomological Shafarevich conjecture by taking into account of the unramifiedness of the cohomology groups of degree other than 2, and Theorem \ref{mainthm-cohom} is proved in the general form of Theorem \ref{cohshaf}. As a consequence, we deduce in Corollary \ref{cor:cohshaf} that the full-degree cohomological Shafarevich conjecture holds for all hyper-Kähler varieties of known deformation type, and for all hyper-Kähler varieties of dimension 4.

\subsection{Geometric finiteness of hyper-Kähler varieties of CM type} 
The key ingredient in the proof of Theorem \ref{mainthm} is the construction of a \textit{uniform Kuga--Satake map} (\cite{OrrSkorobogatov}, \cite{Sh17}, see Section \ref{subsec:UniformKS}) and its arithmetic properties,
which allows us to show, as an intermediate step towards Theorem \ref{mainthm}, that there are only finitely many isometry classes of (geometric) Picard lattices of hyper-Kähler varieties arising from  $\Sha_M^{\ess}(F,R)$ (see Theorem \ref{boundedpolarization}). 

More generally, we propose the following conjecture, inspired by yet another conjecture of Shafarevich formulated in \cite{Shafarevich96}.

\begin{conjecture}[Finiteness of geometric Picard lattices]\label{conj:fin-ns}
Let $d$ be an integer and $M$ a deformation type of hyper-Kähler manifolds.  Then there are only finitely many isometry classes in the following set of lattices:
\begin{equation}\label{eq:pic-bound}
\bigg\{\Pic(X_\CC)\bigg |~ \begin{array}{lr}
	\hbox{\rm  $X$~is hyper-Kähler over a number field $K$ of  degree $\leq d$  } &  \\
	\hbox{\rm   $X_\CC\coloneqq X\times_K \Spec(\CC)$ is of deformation type $M$ for some $K\hookrightarrow \CC$} &
	\end{array}  \bigg \}.
\end{equation} 
Here the Picard group is equipped with the restriction of the Beauville--Bogomolov quadratic form.
\end{conjecture}

As a special case, if we fix the number field $K$ in \eqref{eq:pic-bound},  the finiteness of isometry classes in \eqref{eq:pic-bound} can be regarded as a strengthening of the Bombieri--Lang conjecture for moduli spaces of lattice-polarized hyper-Kähler varieties, which predicts that the $K$-rational points in a moduli space  $\scF$ of lattice-polarized hyper-Kähler varieties are contained in the Noether--Lefschetz loci if $\scF$ is of ``sufficiently" general type. 

One could also speculate the more ambitious form without the restriction on deformation type. In the case of K3 surfaces, Conjecture \ref{conj:fin-ns} has been confirmed for K3 surfaces of CM type by Orr--Skorobogatov in \cite[Corollary B.1]{OrrSkorobogatov}. In fact, they proved the following stronger result \cite[Theorem B]{OrrSkorobogatov}:

\begin{theorem}[Orr--Skorobogatov]\label{thm:OrrSkoro}
There are only finitely many geometric isomorphism classes of K3 surfaces of CM type which can be defined over a number field of a given degree. 
\end{theorem}

As another application of the uniform Kuga--Satake construction, our second main result generalizes Theorem \ref{thm:OrrSkoro} to hyper-Kähler varieties of CM type. 

\begin{theorem}\label{mainthm3}
Let $d$ be a positive integer and $M$ a fixed deformation type of hyper-Kähler varieties with $b_2 \geq 4$. Then there are only finitely many geometrically birational equivalence classes of hyper-Kähler varieties of CM type and of deformation type $M$ which can be defined over a number field of degree $\leq d$.
When $b_2\geq 5$, there are only finitely many geometric isomorphism classes of hyper-Kähler varieties of CM type and of deformation type $M$ which can be defined over a number field of degree $\leq d$.
\end{theorem}
Note that (geometrically) birational transformations preserve the (geometric) Picard lattice. In particular, Conjecture \ref{conj:fin-ns} holds for hyper-Kähler varieties of CM type:

\begin{corollary}
\label{cor:FinitenessPicard}
Let $d$ be a positive integer and $M$ a fixed deformation type of hyper-Kähler varieties with $b_2 \geq 4$. There are only finitely many  
isometry classes of the geometric Picard lattices of hyper-Kähler varieties of CM type and of deformation type $M$ which can be defined over a number field of degree $\leq d$.
\end{corollary}

As a consequence, we get the uniform boundedness of Brauer groups for hyper-Kähler varieties of CM type. 

\begin{corollary}\label{CorBr}
Fix a deformation type $M$ of hyper-Kähler varieties with $b_2 \geq 4$.
For any positive integer $d$,  there exists a constant $N$ such that
\begin{center}
    $|\Br(X)/\Br_0(X)| < N$ and $|\Br(X_{\bar{\QQ}})^{\Gal(\bar{\QQ}/F)}| < N$
\end{center}
for any hyper-Kähler variety $X$ of CM type, of deformation type $M$, and defined over a number field $F$ of degree at most $d$. Here $\Br_0(X)\coloneqq \operatorname{Im}\left(\Br(F)\to \Br(X)\right)$.
\end{corollary}

\noindent{\bf Structure of the paper.}
In Section \ref{sec:HKgenerality}, we recall some basic notions on hyper-Kähler varieties over a field of characteristic zero, including the ($\widehat{\ZZ}$-)Beauville--Bogomolov form, deformation types, polarization in families, (birational) automorphisms, wall divisors, (birational) ample cones, etc. 
In Section \ref{sec:ShafarevichSets}, we introduce various Shafarevich conditions and the corresponding Shafarevich sets. 
In Section \ref{sec:ModuliHK}, we define the moduli stacks/spaces of (oriented) polarized hyper-Kähler varieties (with level structures). 
In Section \ref{sec:ArithPerMap}, we first follow Bindt \cite{BindtThesis} to study the period map from the moduli spaces of hyper-Kähler varieties to Shimura varieties of orthogonal type, and then develop a uniform Kuga--Satake construction generalizing She \cite{Sh17}. 
In Section \ref{sectionMatsusakaMumford}, we show the finiteness of twists with smooth reductions.  
In Section \ref{sec:FinitnessResults}, we prove the unpolarized Shafarevich conjecture (Theorem \ref{mainthm}). 
In Section \ref{sectionremcoh}, we establish a cohomological version of the unpolarized Shafarevich conjecture (Theorem \ref{mainthm-cohom}). 
In Section \ref{sec:FiniteCM}, we show the finiteness of CM type hyper-Kähler varieties (Theorem \ref{mainthm3}) and deduce Corollary \ref{CorBr}.
In Appendix \ref{appendixMatsusakaMumford}, we provide a generalization to the setting of algebraic spaces of Matsusaka--Mumford's theorem on specializations of birational maps.

\subsection*{Acknowledgment}  We would like to thank Fran\c{c}ois Charles, Brendan Hassett, Ariyan Javanpeykar, Tatsuro Kawakami, Giovanni Mongardi, Ben Moonen, Chenyang Xu, Shou Yoshikawa, Zhiwei Yun for helpful suggestions and valuable discussions. T.T. owes deep gratitude to his advisor Naoki Imai for encouragement and helpful advice and to Tetsushi Ito and Yoichi Mieda for helping him with various parts of the paper, including the proof of Theorem \ref{cohshaf} and Lemma \ref{Lemma:Specialization}. 

\section{Generalities on hyper-Kähler varieties}
\label{sec:HKgenerality}
Hyper-Kähler varieties are higher-dimensional analogues of
$K3$ surfaces. They play a significant role in algebraic geometry since they form one type of building blocks of varieties with vanishing first Chern class, by the Beauville--Bogomolov decomposition theorem (\cite{Bea83, Bogomolov74}). 
In this section, we collect some generalities on these varieties. See \cite{Bea83, MR1664696, MR1963559} for basic results and examples over the field of complex numbers.

Firstly, we fix the definition: a {\it  hyper-Kähler} (or  {\it irreducible symplectic}) variety over a field $k$ of characteristic $0$ is a geometrically connected smooth projective $k$-variety $X$ such that it is geometrically simply connected $\pi_1^{\et}(X_{\bar{k}})=1$, and $\rH^0(X,\Omega^2_{X/k})$ is spanned by a nowhere degenerate algebraic
$2$-form. Note that the $2$-form is automatically closed (hence symplectic) by the degeneration of the Hodge-de Rham spectral sequence. Moreover, $\pi_1^{\et}(X_{\bar{k}})=1$ implies that $\rH^1_{\et}(X_{\bar{k}},\ZZ_{\ell})=0$ by \'etale Hurewicz theorem. Using Beauville--Bogomolov decomposition theorem, it is easy to see that $X$ is hyper-Kähler if and only if $X\times_k \Spec\CC$ is a complex hyper-Kähler manifold in the sense of \cite{Bea83, MR1664696}, for any embedding $k\hookrightarrow \CC$; see \cite[Lemma 3.1.3]{BindtThesis}.

\subsection{Beauville--Bogomolov form} 
\label{subsec:BBform}
For a complex hyper-Kähler variety $X$, $\rH^2(X, \ZZ)$ carries an integral, primitive quadratic form $q_X$,  called the Beauville--Bogomolov (BB) form which satisfies the following conditions:

\begin{enumerate}
	\item $q_X$ is non-degenerate and of signature $(3,b_2(X)-3)$.
	\item  There exists a positive rational number $c_X$, called the {\it Fujiki constant}, such that $q_X^n(\alpha)=c_X\int_X \alpha^{2n}$ for all classes $\alpha\in \rH^2(X,\ZZ)$, where $2n=\dim(X)$.
	\item The Hodge decomposition $\rH^2(X, \CC)=\rH^{2,0}(X)\oplus \rH^{1,1}(X)\oplus H^{0,2}(X)$ is orthogonal with respect to $q_X\otimes \CC$. 
\end{enumerate}

\begin{remark}
\label{rmk:evenness}
	All known examples of complex hyper-Kähler manifolds have an even BB form, that is, $\Lambda_M$ is an even lattice.  The evenness of $\Lambda_M$ in general is unknown. Let us mention that there are examples of hyper-Kähler \textit{orbifolds} with odd BB form \cite{KapferMenet}.
\end{remark} 

By Bogomolov--Tian--Todorov \cite{Bo78}, the deformation space of $X$ is unobstructed and its smooth deformations remain hyper-Kähler. The property $(2)$  ensures the deformation invariance of the Fujiki constant $c_X$ and the lattice $\Lambda_X\coloneqq(\rH^2(X,\ZZ), q_X)$. From now on, we assume that $k$ is finitely generated. 

For a hyper-Kähler variety $X$ defined over a field $k$ of characteristic zero, upon fixing an embedding $\iota\colon k\hookrightarrow \CC$, Artin's comparison isomorphism
\begin{equation}
\rH^2(X\times_{k, \iota}\Spec\CC,\ZZ(1))\otimes \widehat{\ZZ}\cong  \rH^2_{\et} (X_{\bar{k}},\widehat{\ZZ}(1))
\end{equation}
allows us to transport the BB form $q_X$ on $\rH^2(X\times_{\iota}\Spec\CC, \ZZ(1))$ to a $\widehat{\ZZ}$-quadratic form:
$$q\colon\rH^2_{\et} (X_{\bar{k}},\widehat{\ZZ}(1))\times \rH^2_{\et} (X_{\bar{k}},\widehat{\ZZ}(1)) \to \widehat{\ZZ},$$ where $\bar{k}$ denotes the algebraic closure of $k $ in $\CC$.
As shown in \cite[Lemma 4.2.1]{BindtThesis}, this $\widehat{\ZZ}$-quadratic form is independent of the choice of embedding, which is called the \textit{$\widehat{\ZZ}$-BB form} of $X$. Similarly, the Fujiki constant $c_X$ is defined to be the Fujiki constant of $X\times_{\iota}\Spec\CC$, which is also independent of $\iota$.

\subsection{Deformation type and \texorpdfstring{$\widehat{\ZZ}$}{ℤ̂}-numerical type}
\label{subsec:DeformationType}
\begin{definition}[Deformation type]
Let $M$ be a fixed complex hyper-Kähler manifold.
Given an embedding $\iota\colon k\hookrightarrow \CC$, we say that a hyper-Kähler variety $X$ over $k$ is of (complex) \textit{deformation type} $M$ with respect to $\iota$ if the complex variety $X_\CC \coloneqq X\times_{\iota} \Spec (\CC)$ is deformation equivalent to $M$. We say that a hyper-Kähler variety $X$ is of deformation type $M$ if it is so for some embedding $k\hookrightarrow \CC$. 
\end{definition}

So far, only a few deformation types have been discovered: in each even dimension $2n$, there are the $K3^{[n]}$-type and the generalized Kummer type $\operatorname{Kum}_n$ constructed by Beauville \cite{Bea83}, and O'Grady constructed in \cite{OG6} and \cite{OG10} two other sporadic examples in dimension 6 and 10, called $OG_6$-type and $OG_{10}$-type respectively.  Clearly, a hyper-Kähler variety can have at most finitely many deformation types. It is an interesting open question whether the deformation type is actually independent of the embedding $\iota\colon k\hookrightarrow \CC$. For hyper-Kähler varieties of known types, one can obtain the following result.

\begin{proposition}
Let $k$ be a  field of characteristic $0$. If there exists $\iota_0:k\hookrightarrow \CC$ such that $X\times_{\iota_0} \Spec (\CC)$ is one of the four known deformation types, then $X\times_{\iota}\Spec(\CC)$ is of the same deformation type for any embedding $\iota:k\hookrightarrow \CC$.
\end{proposition}
\begin{proof}
This was proved in \cite[Proposition 2.3.1]{Yang2019} for $X\times_{\iota}\Spec(\CC)$ of the $K3^{[n]}$ type, and the argument is similar for other known deformation types. Let us sketch the proof for the convenience of the reader. We write $X_\CC=X\times_{\iota_0} \Spec(\CC)$. By the Lefschetz principle, one can assume that $k$ is finitely generated over $\mathbb{Q}$.

    By Mongardi and Pacienza \cite[Corollary 1.2 \& Example 3.18]{MongardiPacienza}, $X_\CC$ is deformation equivalent via a chain of projective families to $Y$, a crepant resolution of the albanese fiber of the moduli space of semistable vector bundles on some smooth K3 or abelian surface $S$ with Mukai vector of the form $(1, 0, n)$ or $(2,0,-2)$, with respect to some suitable generic polarization.
  
   For any embedding $\iota: k\hookrightarrow \CC$, the isomorphism $\iota(k)\cong \iota_0(k)$ can be extended to an element in $\Aut(\CC)$, which we still denote by $\iota$. By conjugating the chain of projective deformations between $X$ and $Y$, we see that $X\times_\iota \Spec(\CC)$ is of the same deformation type as $Y\times_\iota \Spec(\CC)$.
   
   The conjugate surface $S_{\iota}=S\times_\iota \Spec(\CC)$ remains a K3 or abelian surface. The conjugate $Y\times_\iota \Spec(\CC)$ is then birational to a crepant resolution of the albanese fiber of the moduli space of semistable vector bundles on $S_{\iota}$ with the same Mukai vector, with respect to some suitable generic polarization. 

 Thanks to Huybrechts' result \cite[Theorem 4.6]{MR1664696},  the geometric deformation type is invariant under birational transformations between hyper-Kähler varieties. Therefore the deformation type of $X\times_\iota \Spec(\CC)$ is the same as that of $X_\CC$. 
\end{proof}

If $X$ is of deformation type $M$, denoting  $\Lambda_M$  the lattice realized by the BB form on $\rH^2(M,\ZZ)$  and $c_M$ its Fujiki constant, then as we mentioned above in Section \ref{subsec:BBform}, $\Lambda_M\otimes \widehat{\ZZ}$ is the $\widehat{\ZZ}$-BB form of $X$, and $c_M$ is the Fujiki constant of $X$, regardless of the embedding $\iota$.

In \cite{GHS10}, the authors consider the moduli problem of polarized hyper-Kähler manifolds of fixed numerical type. This motivates us to introduce the following weaker notion of $\widehat{\ZZ}$-\textit{numerical type} of hyper-Kähler varieties over general fields, which is more appropriate for our analysis when addressing moduli spaces and period maps.

 \begin{definition}[$\widehat{\ZZ}$-numerical type]
We fix a positive integer $n$.
Let $k$, $k'$ be fields of characteristic $0$, and $X$ (resp.\,$X'$) a $2n$-dimensional hyper-Kähler variety over $k$ (resp.\,$k'$).
We say that $X$ and $X'$ are \emph{$\widehat{\ZZ}$-numerically equivalent} if there exists an isomorphism of $\widehat{\ZZ}$-modules
\[
g\colon
\rH^{2}_{\et}(X_{\overline{k}},\widehat{\ZZ}) \simeq \rH^{2}_{\et} (X'_{\overline{k'}}, \widehat{\ZZ})
\]
such that 
$g$ is isometric with respect to the $\widehat{\ZZ}$-BB forms, and $X$ and $X'$ have the same Fujiki constant.
\end{definition}

\begin{remark}
\label{remnumdef}
A $\widehat{\ZZ}$-numerical type is the union of finitely many deformation types. 
Indeed, if $X$ and $Y$ are of the same deformation type (i.e.~$X_\CC$ and $Y_\CC$ are deformation equivalent), then by the definitions of $\widehat{\ZZ}$-BB form and Fujiki constant, $X$ and $Y$ clearly have the same $\widehat{\ZZ}$-numerical type. Conversely, since there are only finitely many possible isometry classes of $\ZZ$-lattices of bounded rank and discriminant, fixing a $\widehat{\ZZ}$-numerical type $N$ leaves only finitely many possibilities for the BB lattice of $X_\CC$ for any $X$ of $\widehat{\ZZ}$-numerical type $N$. By Huybrechts \cite[Theorem 4.3]{Hu03}, there are only finitely many possibilities for the deformation types.
\end{remark}

\subsection{Polarization and Picard lattice} 
\label{subsec:Polarization}
A \textit{family} of (projective) hyper-Kähler varieties over a $\QQ$-scheme $T$ means a proper smooth morphism of algebraic spaces $\mathfrak{X}\to T$ such that every geometric fiber is a projective hyper-Kähler variety. 

Given such a family of projective hyper-Kähler varieties $$f\colon \fX \to T,$$
the relative Picard functor $\Pic_{X/T}$ is represented by a separated algebraic space over $T$, denoted by the same notation. In particular, if $X$ is a hyper-Kähler variety over a field $k$ of characteristic zero,  then the relative Picard functor $\Pic_{X/k}$ is represented by a separated, smooth, $0$-dimensional scheme over $k$ (cf.~\cite[\S 8.4, Theorem 1]{BLR90}). 

Let $G_k\coloneqq \Gal(\bar{k}/k)$ be the absolute Galois group of $k$.
Then the group $\Pic_{X/k}(k)$ is identified with the subgroup 
$\Pic_{X/k}(\bar{k})^{G_k} \subset \Pic_{X/k}(\bar{k})=\Pic(X_{\bar{k}})$, and hence is torsion-free (since $\Pic(X_{\bar{k}})$ is so). Denote $\Pic_{X/k}(k)$ by $\Pic_X$ in the sequel.

\begin{definition}[cf.~\cite{A96,Ri06}] 
Let $\fX\to T$ be a family of hyper-Kähler varieties over a $\QQ$-scheme $T$.
A \textit{polarization} on $\fX/T$ is a global section $\mathcal{H} \in  \Pic_{ \fX/T}(T)$, such that for every geometric point $t$ of $T$, the fiber $\mathcal{H}_t\in \Pic_{\fX_t}$ is an ample line bundle on $\fX_t$. 
\end{definition}

Let $X$ be a hyper-Kähler variety over a field $k$ of characteristic zero. For each prime $\ell$, the $\ell$-adic cycle class map defines an embedding
\[
c_1^{\ell} \colon \Pic(X)\otimes \ZZ_\ell \rightarrow \rH^2_{\et}(X_{\bar{k}},\ZZ_\ell(1)),
\]
which factors through the geometric first Chern class map
\begin{equation}
\label{eqn:c1-kbar}
\Pic(X_{\bar{k}}) \otimes \widehat{\ZZ} \to \rH^2_{\et}(X_{\bar{k}}, \widehat{\ZZ}(1)).
\end{equation}
Putting all prime $\ell$ together, we get 
\begin{equation}\label{eqn:c1-Zhat}
c_1 \colon \Pic(X)\otimes \widehat{\ZZ} \rightarrow \rH^2_{\et}(X_{\bar{k}},\widehat{\ZZ}(1)),
\end{equation}
Since the image of $c_1^{\ell}$ is $G_k$-invariant for any prime $\ell$, the morphism \eqref{eqn:c1-Zhat} is extended to a homomorphism (keeping the same notation):
\begin{equation}\label{eqn:c1-extended}
c_1 \colon  \Pic_X \otimes_{\ZZ} \widehat{\ZZ} \hookrightarrow \rH^2_{\et}(X_{\bar{k}}, \widehat{\ZZ}(1)),
\end{equation}
whose image lies in the $G_{k}$-invariant part.

\begin{proposition}\label{prop:TateConjecture}
Let $X$ be a hyper-Kähler variety over a number field $k$ with $b_2>3$. Then the first Chern class map \eqref{eqn:c1-extended} induces an isomorphism:
\[
c_1  \colon \Pic_X \otimes \widehat{\ZZ} \xrightarrow{\sim} \rH^2_{\et}(X_{\bar{k}}, \widehat{\ZZ}(1))^{G_k}.
\]
\end{proposition}
\begin{proof}
When $b_2(X)>3$, the Tate conjecture for divisors on polarized hyper-Kähler varieties over number fields holds true, i.e., for each prime $\ell$, the first Chern class map induces an isomorphism of $\QQ_{\ell}$-vector spaces
\[
c_1^{\ell} \otimes \QQ_{\ell} \colon \Pic_X \otimes \QQ_{\ell} \xrightarrow{\sim} \rH^2_{\et}(X_{\bar{k}}, \QQ_{\ell})^{G_k}.
\]
See \cite[Theorem 1.6.1 (2)]{A96}.
Therefore $\rk (\Pic_X \otimes_{\ZZ} \widehat{\ZZ}) = \rk \left(\rH^2_{\et}(X_{\bar{k}}, \widehat{\ZZ}(1))^{G_k}\right)$. To conclude, it suffices to show that $\Pic_{X} \otimes_{\ZZ} \widehat{\ZZ}$ is saturated in $\rH^2_{\et}(X_{\bar{k}}, \widehat{\ZZ}(1))$ via $c_1$.

To this end, by the exponential sequence, the cokernel of the (complex) first Chern class map 
\begin{equation}
    \label{eqn:c1-Complex}
    c_1^{\CC}\colon \Pic(X_{\bar{k}}) \cong \Pic(X_{\CC}) \hookrightarrow \rH^2(X(\CC), \ZZ)
\end{equation}
lies in $\rH^2(X_{\CC}, \mathcal{O}_{X_{\CC}})$, which is torsion-free. Thus $\Pic(X_{\bar{k}}) \subseteq  \rH^2(X(\CC), \ZZ)$ is saturated.
The complex and the $\widehat{\ZZ}$-first Chern class maps \eqref{eqn:c1-kbar} and \eqref{eqn:c1-Complex} are compatible in the sense that the comparison isomorphism 
$\rH^2(X(\CC), \ZZ)\otimes \widehat{\ZZ}\cong \rH^2_{\et}(X_{\bar{k}}, \widehat{\ZZ})$ 
identifies the $\widehat{\ZZ}$-sublattices 
\begin{equation}
c_1^\ell(\Pic(X_{\bar{k}})\otimes \widehat{\ZZ}) = c_1(\Pic(X_{\bar{k}}))\otimes \widehat{\ZZ}.
\end{equation}
Therefore, the sub-lattice $\Pic(X_{\bar{k}})\otimes \widehat{\ZZ} \subseteq  \rH^2_{\et}(X_{\bar{k}}, \widehat{\ZZ}(1))$ is saturated.
Since $\Pic_X \subseteq \Pic(X_{\bar{k}})$ is also saturated by definition, $\Pic_X\otimes \widehat{\ZZ} \subseteq  \rH^2_{\et}(X_{\bar{k}}, \widehat{\ZZ}(1))$ is saturated. 
\end{proof}

\subsection{Automorphisms and birational self-maps}\label{subsect:AutBir}
Let $X$ be a hyper-Kähler variety over a finitely generated field $k$ of characteristic zero.  For any field extension $k'/k$, denote by $\Aut(X_{k'})$ the group of $k'$-automorphisms of $X_{k'}$  and $\Bir(X_{k'})$ the group of $k'$-birational self-maps. The group $\Aut(X_{k'})$ is discrete.
For an embedding $k\hookrightarrow \CC$, the action of $\Aut(X_{\CC})$ on the second cohomology preserves the Beauville--Bogomolov form. Define 
\begin{equation}
    \Aut_0(X_\CC)\coloneqq\ker\left(\Aut(X_\CC)\to \mathrm{O}(\rH^2(X_\CC,\ZZ))\right);
\end{equation}
and its subgroup of automorphisms acting trivially on the whole cohomology:
\begin{equation}
    \Aut_{00}(X_\CC)\coloneqq \ker\left(\Aut(X_\CC)\to \GL(\rH^*(X_\CC,\QQ))\right).
\end{equation}

\begin{theorem}[{Huybrechts \cite[Proposition 9.1]{MR1664696}, Hassett--Tschinkel \cite[Theorem 2.1]{HT13}}]
\label{thm:Aut0}
 The group $\Aut_0(X_\CC)$ is a finite group and it is deformation invariant for hyper-Kähler manifolds.
\end{theorem}

Similarly, we have the following result concerning $\Aut_{00}$. 
\begin{proposition}
\label{prop:Aut00}
 The group $\Aut_{00}(X_\CC)$ is invariant under deformation for hyper-Kähler manifolds.
\end{proposition}
\begin{proof}
We are in the complex setting and drop the subscript $\CC$ here. 
The first step is similar to Hassett--Tschinkel \cite[Theorem 2.1]{HT13}: given a complex hyper-Kähler manifold $X$, the group $\Aut_{00}(X)$ acts equivariantly on the universal deformation over the Kuranishi space $\mathfrak{X}\to \operatorname{Def}(X)$. Since the action on $$\rH^1(X, T_X)\simeq \rH^1(X, \Omega_X^1)\subset \rH^2(X, \CC)$$ is the identity, the action is the identity on $\operatorname{Def}(X)$, i.e. it is a fiberwise action. This shows that for any family of hyper-Kähler manifolds $\mathfrak{X}\to B$, the group scheme map $\pi\colon \Aut_{00}(\mathfrak{X}/B)\to B$ is a local homeomorphism. We are reduced to showing the universal closedness of $\pi$ for a family $\mathfrak{X}/B$; in other words, given a family of (fiberwise) automorphisms $\phi\in \Aut_0(\mathfrak{X}/B)$, if $\phi_t\in \Aut_{00}(X_t)$ for all $t\neq 0$, then the specialization $\phi_0\in \Aut_0(X_0)$ also acts trivially on the whole cohomology ring. To this end, note that the specialization (as cycles) of the graph $\Gamma_{\phi_t}$ is the graph $\Gamma_{\phi_0}$, and $[\Gamma_{\phi_t}]=[\Delta_{X_t}]$ for all $t\neq 0$, we have that $[\Gamma_{\phi_0}]=[\Delta_{X_0}]$, i.e. $\phi_0$ acts trivially on $\rH^*(X_0, \QQ)$.
\end{proof}

\begin{example}\label{exm:aut}
The group $\Aut_0(X_\CC)$ for a complex hyper-Kähler variety $X_\CC$ of known deformation types has been worked out, and $\Aut_{00}(X_\CC)$ is shown to be trivial for all known deformation types and also in general in dimension 4.
\begin{enumerate}
    \item $K3^{[n]}$-type: $\Aut_0$ is trivial by  Beauville \cite[Proposition 10]{Beauville-82Katata}. 
     \item Generalized Kummer type $\operatorname{Kum}_n$: $$ \Aut_0\cong (\ZZ/(n+1)\ZZ)^4\rtimes \ZZ/2\ZZ,$$  
by Boissi\`ere, Nieper-Wisskirchen, and Sarti \cite[Corollary 5]{BNS11}. Oguiso proved in \cite[Theorem 1.3]{Ogu20}
that $\Aut_{00}$ vanishes. 
\item $OG_6$-type: Mongardi and Wandel proved that $\Aut_0\cong (\ZZ/2\ZZ)^{8}$ in \cite[Theorem 5.2]{MW17}, and that $\Aut_{00}$ vanishes in \cite[Remark 6.9]{MW17}. Note that they showed the vanishing of $\Aut_{00}$ for crepant resolutions of the projective $OG_6$ singular moduli spaces. Then by Proposition \ref{prop:Aut00}, %
we can conclude the vanishing for all $OG_6$-type manifolds; cf.~\cite[Section 4]{Ogu20}.
    \item $OG_{10}$-type: $\Aut_0$ is trivial by Mongardi and Wandel \cite[Theorem 3.1]{MW17}.
 \item Dimension $4$: $\Aut_{00}$ is trivial by Jiang and Liu \cite{JiangLiu-NumericalTrivialAutomorphism}.
\end{enumerate}
\end{example}

Moreover, we can consider the automorphism group of polarized hyper-Kähler varieties. For a polarized pair $(X,\xi)$, we define  $$\Aut\left(X_{k'},\xi_{k'}\right)\coloneqq \Set*{f\in \Aut(X_{k'})\given f^\ast \xi=\xi},$$ which is finite by the Matsusaka--Mumford theorem (see \cite[\S 1, Corollary 2]{MM64}).  Similarly, one could also consider the group  $\Bir(X_{k'},\xi_{k'})$ of $k'$-birational self maps of $(X_{k'},\xi_{k'})$. It is well-known that between two smooth projective varieties with trivial canonical bundles, any birational map that preserves an ample class is an isomorphism, i.e., $\Bir(X_{k'},\xi_{k'})=\Aut(X_{k'},\xi_{k'})$ (see \cite[Corollary 3.3, Lemma 3.4]{Fuj81} for the geometric case).

\subsection{Cone conjecture and finiteness of birational models}
\label{sec:WallDivisors}

Let $X$ be a hyper-Kähler variety over $k$ of a given deformation type. 
Let $N^1(X)=\Pic(X)$ be the  N\'eron--Severi lattice of $X$, equipped with the Beauville--Bogomolov form. Inside the \textit{N\'eron--Severi space} $$N^1(X)_\RR=N^1(X)\otimes \RR,$$ we have the positive cone $\Pos(X)$ defined as the connected component of $\{v\in N^1(X)_\RR~|~ v^2>0\}$ containing the ample classes,  and the nef cone $\Nef(X)$ defined as the closure of the ample cone $\Amp(X)$.  
The \textit{birational ample cone}, denoted by $\BA(X)$, is defined as the union $$\bigcup\limits_f f^\ast \Amp(Y),$$ where $f$ runs through all $k$-birational isomorphisms from $X$ to other hyper-Kähler varieties over $k$.  It is known that the closure of $\BA(X)$ is the movable cone $\overline{\Mov}(X)$ (see \cite{HT13, takamatsu21}).

\begin{definition}(\cite{AV15}, \cite{Mongardi-WallDiv})\label{def:WallDiv}
Let $D$ be a divisor class on a hyper-Kähler variety $X_\CC$. Then $D$ is called a \textit{wall divisor} or a \textit{monodromy birationally minimal} (MBM) class if $q(D) <0$ and $$\Phi(D^\perp) \cap  \BA(X_\CC) =\emptyset, $$ for any
 Hodge monodromy operator $\Phi$. We denote the set of wall divisors on $X_\CC$ by $\mathcal{W}(X_\CC)$.
\end{definition}

In \cite{AV17}, Verbitsky and Amerik proved the following boundedness results on wall divisor classes and confirmed the Kawamata--Morrison cone conjecture for hyper-Kähler varieties (see also \cite{AV-IMRN}).

\begin{theorem}[{\cite[Theorem 1.5, Theorem 1.7]{AV17}}]
\label{con-conj}
Let $X$ be a smooth hyper-Kähler variety with $b_2\geq 5$ over a field $k$ of characteristic zero. Then the Beauville--Bogomolov squares of the elements in $\mathcal{W}(X_\CC)$ are bounded (from below).   
\end{theorem}

Using  Markman and Yoshioka's argument in \cite[Corollary 2.5]{MY15} and the movable cone conjecture for complex hyper-Kähler varieties (proved by Markman \cite{Mar11}), the boundedness of squares of wall divisors above allow Amerik and Verbitsky to conclude the finiteness of the set of birational hyper-Kähler models of a given hyper-Kähler variety over the complex numbers. Over an arbitrary field of characteristic zero, using the method of Bright--Logan--van Luijk \cite{BLvL} for K3 surfaces, the third author \cite[Theorem 4.2.7]{takamatsu21} generalized Markman's work \cite{Mar11} and deduced in the similar way the finiteness of birational hyper-Kähler models:

\begin{theorem}[Finiteness of birational models]\label{cor:bir-fin}
Let $k$ be a field of characteristic zero. Let $X$ be a hyper-Kähler variety defined over $k$ with $b_2\geq 5$. 
Then there are only finitely many $k$-isomorphism classes of hyper-Kähler varieties defined over $k$ which are $k$-birational to $X$.
\end{theorem}

\section{Various Shafarevich sets of hyper-Kähler varieties} 
\label{sec:ShafarevichSets}

In this section, we first define in a more general setting of the Shafarevich sets \eqref{HKset}, \eqref{homHKset} mentioned in the introduction, and then introduce the notion of essentially good reduction and the corresponding Shafarevich set. Throughout the section, let $R$ be a finitely generated $\ZZ$-algebra which is a normal integral domain with fraction field $F$. Fix positive integers $n$ and $d$. Let $M$ be a $\widehat{\ZZ}$-numerical type or a deformation type of $2n$-dimensional hyper-Kähler varieties (Section \ref{subsec:DeformationType}); these two choices do not make much difference thanks to Remark \ref{remnumdef}, and we will make use of both.

\subsection{Polarized and unpolarized Shafarevich sets and their cohomological versions}
The Shafarevich set \eqref{HKset} is generalized as follows:
\begin{definition}[Polarized and unpolarized Shafarevich sets]\label{def:ShafaSets}
Let $M$ be a deformation type of hyper-Kähler variety. 
\begin{equation}\label{unpolHKset}
\Sha_{M}(F,R)=\Set*{X\given 
\parbox{19em}{
$X$ is a hyper-Kähler variety of type $M$ over $F$, \\
 for any height $1$ prime ideal $\mathfrak{p}\in \Spec R$, $X$ has good reduction at $\mathfrak{p}$
}
}/ \cong_{F}.
\end{equation}
Furthermore, let $d$ be a positive integer.
\begin{equation*}
\Sha_{M,d}(F,R)=\Set*{ (X, \xi) \given \parbox{16em}{
$(X, \xi)$ is a polarized hyper-Kähler variety of type $M$ over $F$ with $(\xi)^{2n}=d$, \\
for any height $1$ prime ideal $\mathfrak{p}\in \Spec R$, 
 $X$ has good reduction at $\mathfrak{p}$
}
}/ \cong_{F}.
\end{equation*}
Here, $X$ has \textit{good reduction} at $\mathfrak{p}$ means that there is a regular algebraic space $\fX$ which admits a smooth and proper morphism to the spectrum of the local ring $R_{\mathfrak{p}}$ such that $\fX_{F} \cong X$ as $F$-schemes.
\end{definition}

Following \cite{Takamatsu2020K3}, we define the following \emph{cohomological (polarized) Shafarevich sets}.
\begin{definition}[Cohomological Shafarevich sets]
Let $M$ be a deformation type of hyper-Kähler variety.
\begin{equation}\label{homunpolHKset}
\Sha_{M}^\textrm{hom}(F,R)=\Set*{X~\given \parbox{19em}{
$X$ is a hyper-Kähler variety of type $M$ over $F$, \\
for any height $1$ prime ideal $\mathfrak{p}\in \Spec R$, $\rH^2_{\et}(X_{\bar{F}} , \QQ_\ell)$ is unramified at $\mathfrak{p}$ for some $\ell\neq \operatorname{char}(k(\mathfrak{p}))$}
}/\cong_{F}.
\end{equation}
Furthermore, let $d$ be a positive integer.
\begin{equation}\label{hompolHKset}
\Sha_{M,d}^\textrm{hom}(F,R)=\Set*{ (X,\xi)\given 
\parbox{18em}{
$(X, \xi)$ is a polarized hyper-Kähler variety of degree $d$ of type $M$ over $F$,\\
and for any height $1$ prime ideal $\mathfrak{p}\in \Spec R$, 
 $ \rH^2_{\et}(X_{\bar{F}} , \QQ_\ell)$ is unramified at $\mathfrak{p}$ for some $\ell\neq \operatorname{char}(k(\mathfrak{p}))$}
}/ \cong_{F}.
\end{equation}

Here, $\rH^2_{\et}(X_{\bar{F}} , \QQ_\ell)$ is \textit{unramified} at $\mathfrak{p}$ if the inertia group $I_{k(\mathfrak{p})}$ acts trivially on it. For $n=1$ (K3 surfaces), the unramifiedness conditions in \eqref{homunpolHKset} and \eqref{hompolHKset} are independent of the choice of $\ell\neq \operatorname{char}(R/\mathfrak{p})$, see \cite[Section 5]{Takamatsu2020K3}. Unfortunately, we do not know the $\ell$-independence of unramifiedness in higher dimensions. Therefore, we also consider the following slightly more restrictive version of  \eqref{homunpolHKset} and \eqref{hompolHKset}: for a prime number $\ell\in R^{\times}$:

\begin{equation}\label{homunpolHKset-uniform-l}
\Sha_{M}^{\hom_\ell}(F,R)=\Set*{X \given 
\parbox{19em}{
$X$ is a hyper-Kähler variety of type $M$ over $F$, \\
and for any height $1$ prime ideal $\mathfrak{p}\in \Spec R$,
$\rH^2_{\et}(X_{\bar{F}} , \QQ_\ell)$ is unramified at $\mathfrak{p}$
}}/ \cong_{F};
\end{equation}

\begin{equation}\label{hompolHKset-uniform-l}
\Sha_{M,d}^{\hom_\ell}(F,R)=\Set*{(X,\xi)\given 
\parbox{18em}{$(X,\xi)$ is a polarized hyper-Kähler variety of type $M$ over $F$ with $(\xi)^{2n} =d$,\\ and for any height $1$ prime ideal $\mathfrak{p} \in \Spec R$, $\rH^2_{\et}(X_{\bar{F}},\QQ_{\ell})$ is unramified at $\mathfrak{p}$.}
}/ \cong_{F}.
\end{equation}

\end{definition}

Note that the  unpolarized Shafarevich sets $\Sha_M(F,R)$ admits a natural map from the union of polarized Shafarevich sets by forgetting the polarization: 
	\begin{equation}\label{up=p}
	\coprod\limits_{d \in \ZZ_{>0}}\Sha_{M,d}(F,R) \twoheadrightarrow \Sha_M(F,R)
	\end{equation} 
which is surjective.
Similarly for the cohomological Shafarevich sets $\Sha^{\hom}$ and $\Sha^{\hom_\ell}$.

By the smooth and proper base change theorems (see \cite[Th\'eor\`eme XII.5.1]{SGA4} and \cite[{Tag 0DG2}]{stacks-project}), having good reduction implies the unramifiedness. Therefore,  $$\Sha_M(F,R)\subset \Sha_M^{\hom_\ell}(F,R)\subset \Sha_M^{\hom}(F,R);$$ $$\Sha_{M,d}(F,R)\subset \Sha_{M,d}^{\hom_\ell}(F,R)\subset \Sha_{M,d}^{\hom}(F,R).$$
The difference between $\Sha_M(F,R)$ and $\Sha_M^{\hom}(F,R)$ is partially explained by the following observation:
\begin{proposition}\label{prop:bir-obs}
 Let $X$ and $Y$ be hyper-Kähler varieties over $F$. Suppose that $X$ and $Y$ are birationally equivalent over $F$.  Then $X\in \Sha^\mathrm{hom}_{M}(F,R)$ if and only if $Y\in \Sha^\mathrm{hom}_{M}(F,R)$. 
\end{proposition}

\begin{proof}
For any height 1 prime ideal $\mathfrak{p}$ of $R$, let $F_{\mathfrak{p}}$ be the completion of $F$ at $\mathfrak{p}$.
The birational isomorphism between $X$ and $Y$ is an isomorphism in codimension 1, hence we have an isomorphism 
$$\rH^2_{\et}(X_{\bar{F}_{\mathfrak{p}}},\QQ_\ell(1))\cong \rH^2_{\et}(Y_{\bar{F}_{\mathfrak{p}}},\QQ_\ell(1))$$
as $\Gal(\bar{F}_{\mathfrak{p}}/F_{\mathfrak{p}})$-modules for all $\ell$ and $\mathfrak{p}$.  It follows that  $\rH^2_{\et}(X_{\bar{F}_{\mathfrak{p}}},\QQ_\ell)$  is unramified at $\mathfrak{p}$ if and only if $\rH^2_{\et}(Y_{\bar{F}_{\mathfrak{p}}},\QQ_\ell)$ is unramified at $\mathfrak{p}$. 
\end{proof}

In other words, the unramifiedness condition cannot distinguish birationally equivalent objects in $\Sha_M^{\hom}(F,R)$. This suggests our approach towards the Shafarevich conjecture: we have the finiteness of isomorphism classes within a birational equivalence class (this is guaranteed by \cite{takamatsu21}, recalled above as \Cref{cor:bir-fin}); on the other hand, we study the finiteness of birational equivalence classes in various Shafarevich sets, which is the main goal of the paper, accomplished in \Cref{sec:FinitnessResults}).

\subsection{Shafarevich sets of hyper-Kähler varieties with essentially good reduction} To take into account the birational transformations, we introduce the following notion of \textit{essentially good reduction} of hyper-Kähler varieties, which roughly means having good reduction after birational modification over an unramified base change. 

\begin{definition}[Essentially good reduction]
\label{essentially good}
Let $X$ be a hyper-Kähler variety over $F$.
Let $\p$ be a height $1$ prime ideal of $R$.
We say $X$ has \textit{essentially good reduction} at $\p$ if
there exists a finite \'{e}tale extension $S$ of the completion $\widehat{R}_{\p}$ with respect to the maximal ideal, such that there exists a smooth proper algebraic space $\mathcal{Y}$ over $S$ whose generic fiber $\mathcal{Y}_{\Frac(S)}$ is a hyper-Kähler variety that is $\Frac(S)$-birationally isomorphic to $X_{\Frac(S)}$.
\end{definition}

\begin{remark}
The requirement on base change is meaningful: already in dimension 2, there are examples of K3 surfaces over $\QQ_{p}$ which does not admit good reduction but admits good reduction after a finite unramified extension of $\QQ_{p}$ (see \cite[Theorem 1.6]{LM18}).
\end{remark}

\begin{definition}[Essential Shafarevich sets]
Keep the notation as above.
\begin{equation}\label{essunpolHKset}
\Sha_{M}^{\ess}(F,R)=\Set*{ X\given 
\parbox{19em}{
$X$ is a hyper-Kähler variety of type $M$ over $F$,\\
 for any height $1$ prime ideal $\mathfrak{p}\in \Spec R$, $X$ has essentially good reduction at $\mathfrak{p}$}
}/ \cong_{F}
\end{equation}

\begin{equation}\label{esspolHKset}
\Sha_{M,d}^{\ess}(F,R)=\Set*{ (X, \xi)\given 
\parbox{19em}{
$(X, \xi)$ is a polarized hyper-Kähler variety of type $M$ over $F$, with
$(\xi)^{2n}=d$ \\
for any height $1$ prime ideal $\mathfrak{p}\in \Spec R$, $X$ has essentially good reduction at $\mathfrak{p}$
}
}/ \cong_{F}
\end{equation}
\end{definition}

By definition, we have 
\begin{equation}
            \Sha_{M}(F,R) \subset \Sha_{M}^{\ess}(F,R)\subset \Sha_{M}^{\hom_\ell}(F,R)\subset \Sha_{M}^{\hom}(F,R),
\end{equation}
and similarly for polarized Shafarevich sets. 
The difference between the middle two sets is partially explained in Section \ref{sectionremcoh}.

\section{Moduli of polarized hyper-Kähler varieties}
\label{sec:ModuliHK}
In this section, we review the moduli theory of polarized hyper-Kähler varieties. A good reference is Bindt's PhD thesis \cite{BindtThesis}.
\subsection{Moduli stacks of polarized hyper-Kähler varieties}
Given positive integers $n$ and $d$, consider the groupoid fibration $\scF_d\to (\mathtt{Sch}/\QQ)$ of $2n$-dimensional polarized hyper-Kähler varieties of degree $d$. Let us write it in the form of a moduli functor $$\scF_{d} \colon (\mathtt{Sch}/\QQ)^{op}\to \mathtt{Groupoids}$$  defined as follows: for any $\QQ$-scheme $T$,
\begin{equation}
	\scF_d(T)= \left\{ (f \colon\fX \rightarrow T, \xi)~\Bigg| \begin{array}{lr}
		\hbox{\rm $\fX \rightarrow T$ is a family of $2n$-dimensional hyper-Kähler varieties,}\\
		\hbox{\rm   $\xi\in \Pic_{\fX/T}(T)$ is a polarization with $(\xi_t^{2n})=d$ for all $t\in T(\bar{k})$}   	\end{array}
	\right\}
\end{equation}
(see Section \ref{subsec:Polarization} for the notion of polarization), and for any morphism $T_1\to T_2$ of $\QQ$-schemes, the functor $\scF_d(T_2)\to \scF_d(T_1)$ is given by  pulling back families.
Here, an isomorphism between $(\fX_1 \to T, \xi_1)$ and $(\fX_2 \to T, \xi_2)$ in the groupoid $\scF_d(T)$ is a Cartesian diagram
\[
\begin{tikzcd}
 \fX_1 \ar[r,"\psi","\sim"'] \ar[d,""] & \fX_2 \ar[d, ""] \\
 T \ar[r, "\sim"] & T
\end{tikzcd}
\]
together with an isomorphism $\psi^*(\xi_2) \xrightarrow{\sim} \xi_1$ in the groupoid $\Pic_{\fX_1/T}(T)$.

\begin{proposition}[Bindt]
\label{prop:representabilityHyperK}
The moduli functor $\scF_d$ is represented by a separated smooth Deligne--Mumford stack of finite type over $\QQ$, still denoted by $\scF_d$. 
\end{proposition}
\begin{proof}
\cite[Theorem 3.3.2, Lemma 3.3.9]{BindtThesis} shows that  $\scF_d$ is a separated smooth Deligne--Mumford stack locally of finite type over $\QQ$. To show that it is of finite type, applying Koll\'ar--Matsusaka \cite{KM83}, since $(\xi^{2n})=d$ is fixed, there are only finitely many possibilities for the Hilbert polynomial. Hence $\scF_d$ is of finite type.
\end{proof}

We shall introduce the moduli stack $\widetilde{\scF}_d$ of \textit{oriented} polarized hyper-Kähler spaces of degree $d$, as a finite covering of $\scF_d$.
\begin{definition}[Orientation]
Let $T$ be a $\QQ$-scheme, and let $f \colon \fX\to  T$ be a family of hyper-Kähler varieties with $b_2>3$ (Section \ref{subsec:Polarization}). An \emph{orientation} on $\fX/T$ is an isomorphism of sheaves over $S_{\et}$:
\begin{equation}
   \omega \colon \underline{\ZZ/4\ZZ} \xrightarrow{\simeq} \det R^2_{\et} f_\ast \mu_4.
\end{equation}
This pair $(\fX/S,\omega)$ is called a \emph{family of oriented hyper-Kähler varieties} over $T$.
\end{definition}

The following observation in \cite[Lemma 4.3.2]{BindtThesis} shows that the ``orientation" defined here is equivalent to that in Taelman \cite{taelman17}, at least over a normal base.
\begin{lemma}\label{lemma:liftoforientation}
Let $S$ be a normal scheme of finite type over $\QQ$. For any family of oriented hyper-Kähler varieties $(f\colon \fX \to S, \omega)$ there are unique isomorphisms of lisse $\widehat{\ZZ}$-local systems
\[
\omega_{\et} \colon \underline{\widehat{\ZZ}} \xrightarrow{\sim} \det R^2_{\et}f_* \widehat{\ZZ}(1)
\]
on $S$ and 
\[
\omega_{\operatorname{an}} \colon \underline{\ZZ} \xrightarrow{\sim} \det R_{\operatorname{an}}^2f_* \ZZ(1)
\]
on $S_{\CC}$, 
whose reductions on $\ZZ/4\ZZ$ is $\omega$ and $\omega_{\et}|_{S_{\CC}} = \omega_{\operatorname{an}} \otimes \widehat{\ZZ}$.
\end{lemma}

We denote by $\widetilde{\scF}_d$ the moduli stack of oriented polarized hyper-Kähler varieties of degree $d$.  The natural forgetful functor $\widetilde{\scF}_d \to \scF_d$ is a finite \'etale covering by Lemma \ref{lemma:liftoforientation}, which is of degree two. Thus, the stack $\widetilde{\scF}_d$ is a smooth Deligne--Mumford stack of finite type over $\QQ$, by \Cref{prop:representabilityHyperK}.

\begin{remark}
For K3 surfaces, the moduli stack $\scF_d$ can be defined over $\ZZ$ and is smooth over $\ZZ[\frac{1}{2d}]$.  It is natural to ask whether the same assertion holds for higher dimensional hyper-Kähler varieties. This is known to be true for some families, such as the Fano varieties of lines on cubic fourfolds  \cite{A96}. However, this is a rather difficult problem in general, since the deformation theory and even the ``right'' definition of hyper-Kähler varieties over mixed characteristic fields are unclear at present. See \cite{Yang2019} for the discussion on those of K3$^{[n]}$-type.
\end{remark} 

Different from the case of K3 surfaces, the moduli stack  $\scF_{d}$ is not necessarily (geometrically)  connected.
In practice, as there are only finitely many connected components, we may work with a connected component of the moduli stack $\scF_{d}$ or $\widetilde{\scF}_d$ each time.

\begin{remark}
It is also natural to consider the moduli stack of polarized hyper-Kähler varieties of degree $d$ of a given deformation type, which is a union of some geometric connected components of $\scF_d$. However, we do not know whether such stack is defined over $\QQ$, as it is unclear to us whether the hyper-Kähler varieties in a connected component of $\scF_{d}$ are deformation equivalent. Instead, one can consider the moduli stack $\scF_{M,d}$ of polarized hyper-Kähler varieties of a given $\widehat{\ZZ}$ numerical type $M$ (Section \ref{subsec:DeformationType}), this will be a substack of $\scF_d$ defined over $\QQ$ as $M$ is stable under the $\Gal(\overline{\QQ}/\QQ)$-action. In this way, we can view $\Sha_{M,d}(F,R)$ as a subset of isomorphism classes in $\scF_{M,d}(F)$. 
\end{remark}

\subsection{Level structures}\label{subsection:levelstructure}

Let $\scF_{d}^{\dag}$ be a connected component of $\scF_d$. Choose a point $(X_0,\xi_0) \in \scF_d^{\dag}(\CC)$.
Let $\Lambda$ be the BB lattice of $X_0$. Let $\Lambda_{h}$ be the orthogonal complement of $h= c_1(\xi_0)$ in $\Lambda$.
For any geometric point $(X,\xi)$ in $\scF_{d}^{\dag}(\CC)$, there is an isomorphism of $\widehat{\ZZ}$-lattices
\[
\phi \colon \Lambda \otimes \widehat{\ZZ} \xrightarrow{\sim} \rH^2_{\et}(X, \widehat{\ZZ})  
\]
such that $\phi(h) = c_1(H)$ by the connectedness of $\scF_d^\dag$ (cf.~\cite[Lemma 4.5.1]{BindtThesis}).

We write $G=\SO(\Lambda_{h,{\QQ}})$ and let $K \to  G(\mathbb{A}_f)$ be a continuous group homomorphism from a profinite group, with compact open image and finite kernel. Following \cite{Ri06} or \cite{BL19}, we say that $K\to  G(\mathbb{A}_f)$, or simply $K$, is {\it admissible} if the image of every element of $K$ can be viewed as an isometry of $\Lambda \otimes \mathbb{A}_f$ fixing $h$ and  stabilizing $\Lambda_h \otimes \widehat{\ZZ}$.
Recall that there is an injection
\begin{equation}
K_h \coloneqq \left\{g\in \SO(\Lambda)(\widehat{\ZZ}) \; | \; g(h)=h \right\}\rightarrow G(\mathbb{A}_f).
\end{equation} 
By definition, $K_h$ is admissible, and $K$ is admissible if and only if its image in $G(\mathbb{A}_f)$ lies in $K_h$.

Consider $\mathcal{J}_h$, the sheafification of the
presheaf of sets  on $(\widetilde{\scF}^{\dag}_d)_{\et}$, with sections over a connected $\QQ$-scheme $T$ given by the $\pi^{\et}_1(T,\bar{t})$-invariant set
\begin{equation}
(f \colon \fX \to T, \xi;\omega) \rightsquigarrow \Set*{\alpha \colon \Lambda\otimes\widehat{ \ZZ} \xrightarrow{\sim} \rH^2_{\et}(\fX_{\bar{t}},\widehat{\ZZ}(1)) \given
\begin{matrix}
\text{$\alpha$ is an isometry}\\
\text{such that $\alpha(h)=c_1(\xi_{\bar{t}})$}\\
\text{and $\det(\alpha) = \omega_{\et}$}.
\end{matrix}
}^{\pi_1^{\et}(T, \bar{t})}
\end{equation}
Here $c_1$ is the $\widehat{\ZZ}$-first Chern class maps \eqref{eqn:c1-Zhat}.

By construction, any admissible subgroup $K$ naturally acts on $\mathcal{J}_h(T)$ (via the source of each $\alpha$). Denote $K \backslash\mathcal{J}_h$ for the quotient of $\mathcal{J}_h$ by the left-action of $K$, which is well-defined in the category of \'etale sheaves on $(\widetilde{\scF}_d)_{\et}^{\dag}$.

\begin{definition}  
A \textit{$K$-level structure} on a polarized family of hyper-Kähler varieties $f\colon\fX \rightarrow T$ is a section over $T$ of the \'etale sheaf $K\backslash\mathcal{J}_h$.
\end{definition}

Let $\widetilde{\scF}_{d,K}$ be the moduli stack of oriented polarized hyper-Kähler varieties of degree $d$ with a $K$-level structure. It can be viewed as a finite \'etale cover of the original moduli stack (cf.~\cite[Proposition 3.11]{Pe15}), i.e.,
\begin{proposition}
The stack $\widetilde{\scF}_{d,K}^{\dag}$ is a smooth Deligne--Mumford stack  and the forgetful map 
\begin{equation*}
\pi_{d,K}\colon \widetilde{\scF}_{d,K}^{\dag} \rightarrow \widetilde{\scF}_d^{\dag}
\end{equation*}
is finite and \'etale. 
\end{proposition}

For any field $F$ of characteristic zero, the objects in $\widetilde{\scF}_{d,K}^{\dag}(F)$ are those $(X,H, \omega; \alpha)$ such that
\begin{itemize}
    \item $(X,H, \omega)$ is an oriented polarized family of hyper-Kähler varieties over $F$, and
    \item $\alpha$ is a $K$-level structure $\alpha \colon \Lambda \otimes \widehat{\ZZ} \xrightarrow{\sim} \rH^2_{\et}(X_{\bar{F}},\widehat{\ZZ}(1))$ satisfying
    \[
    \alpha(\rho)(\Gal(\bar{F}/F)) \subseteq K,
    \]
    where $\rho \colon \Gal(\bar{F}/F) \to \Aut{(\rH^2_{\et}(X, \widehat{\ZZ}(1))}$ is the Galois representation,  and $\alpha(\rho)(\sigma) = \alpha^{-1} \circ \rho(\sigma) \circ \alpha$ for any $\sigma \in \Gal(\bar{F}/F)$.
\end{itemize}

In this paper, we mainly consider the following level structures on hyper-Kähler varieties.
\begin{definition}\label{def:levelstructures}
Keep the notation as above.
\begin{enumerate}
   \item The \emph{full-level-$m$ structure} is defined by the admissible group\[K_{h,m} \coloneqq \left\{g \in \SO(\Lambda)(\widehat{\ZZ}) \big|g(h) = h, g \equiv 1~\bmod m\right\} \to G(\mathbb{A}_f).\]
  \item 
 As in \cite{Riz10}, we define the {\it spin level structures} as follows.
Let $\Cl^{+}(\Lambda_h)$ be the even Clifford algebra of $\Lambda_{h,\QQ}$.
Let
\[
\GSpin(\Lambda_h)(\widehat{\ZZ})=\GSpin(\Lambda_h)(\mathbb{A}_f)\cap \Cl^+(\Lambda_h\otimes \widehat{\ZZ}), 
\]
and
\[
K_m^{\operatorname{sp}}= \left\{g \in \GSpin(\Lambda_h)(\widehat{\ZZ}) \big| g \equiv 1~\bmod m\right\}.
\]
Recall that one has an adjoint representation
\[
\ad \colon \GSpin(\Lambda_h) \to G
\]
defined by $ \ad(x) = (v\mapsto xvx^{-1})$. The spin level-$m$ structure is the map
\[
K_m^{\operatorname{sp}} \xrightarrow{\ad} G(\mathbb{A}_f).
\]
We set $K^{\ad}_{m}\subseteq G(\widehat{\ZZ})$ to be the image $K_m^{\operatorname{sp}}$ under the adjoint representation: It is an open compact subgroup of $G(\mathbb{A}_f)$.
\end{enumerate}
For the following usage, we denote by $K^{\ad}_{L,m}$ and $K^{\operatorname{sp}}_{L,m}$  the corresponding level structures for a different lattice $L$.
\end{definition}

It could happen that the action of the automorphism group of a hyper-Kähler variety is not faithful on the second cohomology (see Example \ref{exm:aut}). Therefore, $\widetilde{\scF}_{d,K}$ is not represented by a scheme in this case, even when $K$ is very small, but we have the following remedy. 
\begin{corollary}
If the automorphism group of a very general member of hyper-Kähler varieties parametrized by $\widetilde{\scF}_{d}$ is trivial, then $\widetilde{\scF}_{d,K_{h,m}}^{\dag}$ is represented by a $\QQ$-scheme for $m\geq 3$, still denoted by $\widetilde{\mathcal{F}}_{d,K_{h,m}}^{\dag}$.
\end{corollary}
\begin{proof}It suffices to show that any geometric point $ (X, \xi,\omega; \alpha)$ in $\widetilde{\scF}_{d,K^{\ad}_{m}}^{\dag}(\CC)$ has only trivial automorphism.  Under our assumption, this follows from the fact the automorphism group of $X$ acts faithfully on $\rH^2(X,\ZZ(1))$ and any finite order automorphism of the pair $(\rP^2(X,\ZZ(1)), \alpha)$  is trivial for $m \geq 3$ (cf.~\cite[Lemma 1.5.12]{Ri06}). 
\end{proof}

\section{Arithmetic period map and uniform Kuga--Satake}
\label{sec:ArithPerMap}

In this section, we recall the construction of the arithmetic period map for the moduli space $\widetilde{\scF}_{d}$ of polarized oriented hyper-Kähler varieties of degree $d$. Moreover, we will construct a Kuga--Satake map which is independent of the degree of polarization, which will be called the \emph{uniform Kuga--Satake map} (see \Cref{subsec:UniformKS}). %

\subsection{Orthogonal Shimura varieties} 
We firstly recall the notion of Shimura varieties attached to a lattice. Let $L$ be the lattice over $\ZZ$ of signature $(2, n)$ with $n \geq 1$ and let $G=\SO(L_\QQ)$ be the special orthogonal group scheme over $\QQ$ associated with $L_{\QQ}\coloneqq L\otimes \QQ$. 

Consider the pair $(G,D)$, where $D$ is the space of oriented negative definite planes in $L_{\RR}$:
\[ 
D = \SO(2,n) / \SO(2) \times \SO(n)
\]
where $\SO(2,n) = \SO(L)(\RR)$.
To any admissible level structure $K\to G(\mathbb{A}_f)$, one can associate a Shimura stack $\Sh_K(G,D)$. It is a smooth Deligne--Mumford stack over the reflex field. When $K$ is neat, $\Sh_K(G,D)$ is moreover a smooth quasi-projective variety. The complex points of $\Sh_K(G,D)$ can be identified with the double coset quotient stack:
\begin{equation}\Sh_K(G, D)(\CC)=G(\QQ)\backslash D\times G(\mathbb{A}_f)/K.
\end{equation}
The pair $(G,D)$ is called the Shimura datum of orthogonal type. 
It is well known that the reflex field of the Shimura datum $(G,D)$ is equal to $\QQ$ (see \cite[Appendix]{A96}). 
\begin{remark}
In \cite{taelman17} and \cite{BindtThesis}, the notation $\Sh_K[G,D]$ is used for Shimura stacks to distinguish them from the classical Shimura varieties. For ease of notation, we will use the classical notation to denote the Shimura stack.
\end{remark}

For simplicity, we denote by $K_{L}$ the \textit{discriminant kernel} of $G(\mathbb{A}_f)$, which is the largest subgroup of $G(\widehat{\ZZ})$ that acts trivially on the discriminant of $L$. Its $\ZZ_{p}$-component is just the image of $\GSpin(L)(\ZZ_{p})$ for $p\geq 3$.
In particular, we have
\begin{equation}\label{eq:admissibleinclusion}
    K^{\ad}_{L,m} \subseteq K_{L}
\end{equation}
as a compact open subgroup for $2 \nmid m $, 
where $K^{\ad}_{L,m}$ is defined in Definition \ref{def:levelstructures}.
We simply write \[\Sh(L)\coloneqq \Sh_{K_{L}}(G,D)\]
for the orthogonal Shimura stack with level $K_{L}$, and
\[ 
\Sh_{K}(L) \coloneqq \Sh_K(G,D)
\]
for any open compact subgroup $K \subset K_L$.
For any $m \geq 3$, the inclusion \eqref{eq:admissibleinclusion} induces a finite cover
\[
\Sh_{K^{\ad}_{L,m}}(L) \to \Sh(L)
\]
of degree $[K_{L}: K^{\ad}_{L,m}]$.
If $L$ contains a hyperbolic lattice, the Shimura variety $\Sh(L)(\CC)=\Gamma_{L}\backslash D$ is irreducible and $\Gamma_{L}$ is the largest subgroup of $G(\ZZ)$ that acts trivially on the discriminant (cf.~\cite[(4.1)]{Pe15}).

Let $\underline{L}$ be the local system on $\Sh(L)(\CC)$ attached to the tautological representation $L$ of $\operatorname{O}(L)$. The non-degenerate symmetric bilinear form on $L$ gives rise to an injective map of local systems
$$\underline{L} \rightarrow \underline{L}^\vee.$$ The finite local system $\underline{L}^\vee/\underline{L}$ with its $\QQ/\ZZ$-valued quadratic form is canonically isomorphic to the locally constant sheaf $\underline{\disc(L)}$  over $\Sh(L)(\CC)$. 

The Shimura stack $\Sh(L)$ has the following modular description (cf.~\cite[\S 3]{MilneMotives} or \cite[Lemma 4.4.4]{BindtThesis}).
\begin{proposition}\label{prop:modularshimura}
The groupoid $\Sh(L)(S)$ is the solution of the moduli problem of classifying tuples $(H, q,\xi, \omega,\alpha_{K_{L}})$, where 
\begin{itemize}
		\item  $H$ is a variation of $\ZZ$-Hodge structures over $S$ whose fibers are Hodge lattices of K3-type with signature $(3,m)$;
		\item  $q \colon H \times H \to \underline{\ZZ}$ is a bilinear form of variation of Hodge structures;
		\item  $\xi$ is a global section of $H$ of type $(0,0)$ such that $q(\xi,\xi) >0$;
		\item $\omega \colon \underline{\ZZ} \to \det(H) $ is an isomorphism of local systems;
		\item $\alpha_{K_{L}}$ is a $K_{L}$-level structure on $(H,q, \xi)$; 

\end{itemize}
such that for any point $s \in S$ there is a rational Hodge isometry $\beta_s \colon H_s \otimes \QQ \xrightarrow{\sim } \Lambda \otimes \QQ$  such that $\beta_s(\xi_s) = h$ and preserving the determinant.
\end{proposition}

\subsection{Integral model of Shimura varieties.}\label{subsec:IntegralModel}
Keep the same notation as in the previous section. We recall some results on the existence of integral canonical models of the orthogonal Shimura varieties $\Sh(L)$.

We refer to \cite[Definition 4.2]{Pe16} for the definition of (smooth) integral canonical model. Among the requirements of this definition, the most important one for our usage is the smooth extension property. For the convenience of readers, we record it here.
\begin{definition}\label{def:extensionproperty}
A (pro)-scheme $\mathscr{S}$ over $\ZZ_{(p)}$ is with \emph{(smooth) extension property} if for any regular and locally-healthy (formally smooth) scheme $\mathscr{X}$ over $\ZZ_{(p)}$, any morphism $\mathscr{X}_{\QQ} \to \mathscr{S}$ can be extended to $\mathscr{X} \to \scS$. 
\end{definition}

Then we have the following result.
\begin{theorem}\label{integral-model}
Suppose the discriminant group $\disc(L\otimes \ZZ[\frac{1}{N}])$ is cyclic. Then the Shimura varieties $\Sh(L)$ admits a canonical regular integral models $\scS_{L}$  over $\ZZ[\frac{1}{2N}]$. Moreover, for any neat $K \subset K_{L}$ such that the $p$-primary component $K_p = K_{L,p}$ for some $p \nmid 2N$, there is a finite extension
\[
\scS_{L,K} \to  \scS_{L}.
\]
\'etale on $\ZZ_{(p)}$.
\end{theorem}

\begin{proof}

For any $p \nmid 2N$, the $\ZZ_{(p)}$-lattice $L \otimes \ZZ_{(p)}$ is with cyclic discriminant group.  Let $K_p \subset G(\ZZ_p)$ be the $p$-primary component of the decomposition $K_{L} = K^pK_p$. Applying \cite[Theorem 4.4]{Pe16}, we will obtain an integral model $\scS_{L,K_p} = \{\scS_{L,KK_p}\}_{K \subset K^p}$ over $\ZZ_{(p)}$ as a pro-scheme with \'etale connect morphisms
\[ 
\Sh(L)_{(p)} \coloneqq \Sh_{K_p}(G,D) = \varprojlim_{K \subset K^p} \Sh_{K K_p} (G,D),
\]
 which is formally smooth and with smooth extension property.
 We can consider the union of quotients
\[
\scS_{L} = \bigcup_{p \notin S} \scS_{L,K_p}/{K^p},
\]
which is an integral canonical model of $\Sh(L)$ over $\ZZ[\frac{1}{2N}]$ as required.
\end{proof}

Given $L$ of signature $(2,n)$, suppose we can find a primitive embedding of $\ZZ$-lattices
\[
L \hookrightarrow L_0
\]
such that $L_0$ has signature $(3,n)$.  Let $N_0$ be the product of all primes in $|\mathrm{disc}(L_0)|$. Then $L_0$ is unimodular over $\ZZ[\frac{1}{N_0}]$ and the discrimiant group $\disc \left(L\otimes \ZZ[\frac{1}{N_0}]\right)$ is cyclic.  We deduce the following consequence of Theorem \ref{integral-model}.
\begin{corollary}\label{corollary:integral-shi-hk}
Let $m \geq 3$ be an integer. Let $N$ be the product of all primes in  $2m|\disc(\Lambda)|$. Then $\Sh_{K^{\ad}_m}(\Lambda_h)$ admits a regular integral canonical model over $\ZZ[\frac{1}{2N}]$ for all $\Lambda_h$. 
\end{corollary}

\begin{remark}
The requirement for $\disc(L \otimes \ZZ[\frac{1}{N}])$ being cyclic can be relaxed with some modification of $L$. Actually we can take the regular integral model with parahoric levels given in \cite{KisinPappas,PappasZachos} instead. This will allow us to consider the periods of lattice-polarized hyper-Kähler varieties.
\end{remark}

\subsection{Period map and its descent}

Let $\widetilde{\scF}_{d}^{\dag}$ be a fixed connected component which contains a geometric point $(X_0, H_0)$ as in Section \ref{subsection:levelstructure}. There is an isometry of $\widehat{\ZZ}$-lattices
\[ 
 \phi \colon \Lambda \otimes \widehat{\ZZ} \xrightarrow{\sim}\rH^2_{\et}(X, \widehat{\ZZ})
\]
such that $\phi(h)= \widehat{c}_1(H)$ for any $(X,H) \in \widetilde{\scF}_d^\dag(\CC)$. This $\widehat{\ZZ}$-isometry also induces an isomorphism
\[
\Lambda_h^{\vee}/\Lambda_h \xrightarrow{\sim} \disc(c_1(\xi)^{\bot}).
\]

With the modular description of the orthogonal Shimura stack given in Proposition \ref{prop:modularshimura}, we thus obtain a holomorphic map, called the   \textit{period map}:
\[
\mathcal{P}_{d,K,\CC} \colon \widetilde{\scF}_{d,K}^{\dag}(\CC)\rightarrow \Sh_{K}(\Lambda_h)(\CC), 
\]
which sends $(X,H,\omega,\alpha)$ to $\left((\rP^2_B(X,\ZZ), F^\bullet_{\hdg}),q_X,\omega, c_1(H),\alpha\right)$. 
The global Torelli theorem for polarized hyper-Kähler varieties implies the following. 
\begin{theorem}[{\cite[Theorem 8.4]{Ve13}}]
	The period map 
	\begin{equation}\label{torelli}
	\mathcal{P}_{d,K,\CC}\colon \widetilde{\scF}_{d,K}^{\dag}(\CC)\rightarrow \Sh_{K}(\Lambda_h)(\CC)
	\end{equation}  
	is \'etale.
\end{theorem}

 Let $F$ be a number field. Fix a complex embedding $F \subset \CC$. Let $(X,H)$ be a polarized hyper-Kähler variety over $F$. Let $V$ be its deformation space. Consider its universal deformation $f \colon \mathcal{X} \to V$. 
\begin{proposition}[{\cite[Proposition 16]{Ch13}}]\label{prop:localTorelli}
There is a finite extension $F'$ of $F$ and a formally \'etale morphism of $F'$-schemes $ V \to \Sh(\Lambda_{h})_{F'}$ which factors through $\mathcal{P}_{d,\CC}$ after taking field extension along $F' \subset \CC$.
\end{proposition}

Globally, we have the following result of Bindt. 
\begin{proposition}[{\cite[Theorem 4.5.2]{BindtThesis}}]
\label{prop:descentperiod}
For any  admissible level $K \subset K_{\Lambda_h}$, there is an \'etale morphism over $\QQ$
\[
 \mathcal{P}_{d, K}\colon \widetilde{\scF}_{d,K}^{\dag} \to \Sh_{K}(\Lambda_h),
\]
such that $\mathcal{P}_{d,K} \otimes \CC = \mathcal{P}_{d,K,\CC}$. Moreover, we have a 2-Cartesian diagram
\begin{equation}
\begin{tikzcd}\label{eq:cartisianlevel}
 \widetilde{\scF}_{d,K}^{\dag} \ar[r,"\mathcal{P}_{d, K} "] \ar[d, "\pi_{d,K}"'] & \Sh_{K}(\Lambda_h) \ar[d] \\
 \widetilde{\scF}_{d}^{\dag} \ar[r,"\mathcal{P}_{d}"] &\Sh_{K_h}(\SO(\Lambda_h), D_{\Lambda_h}) .
\end{tikzcd}
\end{equation}
We sometimes write $\mathcal{P}$ for a period map in short.
\end{proposition}

\begin{remark} By Theorem \ref{integral-model}, there exists an integral model for $\Sh_K(\Lambda_h)$. 
However, it is not clear if one  can extend the period map $\mathcal{P}_{d,K}$ integrally as in the case of polarized K3 surfaces.
\end{remark}

\subsection{Kuga--Satake constructions in general}
\label{subsec:KS}
The Kuga--Satake map plays a central role in our approach. The classical Kuga--Satake construction for polarized hyper-Kähler varieties is recalled here. We will construct in the next section a uniform Kuga--Satake map for all polarization types. 

 Given a lattice $L$ of signature $(2, n)$ for $n \geq 1$ and an element $e \in \Cl^+(L)$ that is sent to $-e$ under the canonical involution of the Clifford algebra, the spin representation defines a morphism of Shimura datum
\begin{equation*}
\spin_{e}\colon (\GSpin(L),\widetilde{D}_L)\rightarrow (\GSp(W),\Omega^{\pm})
\end{equation*}
where $W=(\Cl^+(L), \varphi_{e})$ and $\Omega^\pm$ is the Siegel half space. This induces morphisms
\begin{equation}
    \xymatrix{  \Sh_K(\GSpin(L),\widetilde{D}_L)\ar[d]^{\ad}\ar[dr]^{{\spin_e}} &~ \\ \Sh_{K''}(\SO(L), D_L) &  \Sh_{K'}(\GSp(W),\Omega^{\pm}) }
\end{equation}
between the canonical models of Shimura varieties, where $\spin_e(K)\subseteq K'$ and $\ad(K)\subseteq K''$.

If we take the levels $K$ and $K''$ to be the so-called spin levels as in Definition \ref{def:levelstructures}, then there exists a  (non-canonical) section 
\begin{equation}\label{eq:sectionSOSpin}
\gamma_\CC \colon \Sh_{K''}( \SO(L), D_L)(\CC)\rightarrow\Sh_{K}( \GSpin(L), \widetilde{D}_L)(\CC)
\end{equation}
of $\ad$ (cf.~\cite[\S 5.5]{Riz10}).  
It has a descent over a number field $E$, denoted by $\gamma_{E}$. The field $E$ only depends on the group $K$. The composition $$\spin_e(\CC)\circ\gamma_\CC\colon\Sh_K(L)(\CC)\rightarrow \Sh_{K'}(\GSp(W),\Omega^\pm)(\CC) $$ is called the \textit{Kuga--Satake map}. 

Geometrically, write $W_R=W\otimes R$. Then $W_\RR$ is naturally a $G$-module (action by left multiplication), which gives rise to a polarizable Hodge structure of weight one on $W_\ZZ$. This defines a complex abelian variety $A(W)$ of dimension $2^N$, called the \textit{Kuga--Satake variety} attached to $(W_\ZZ, \varphi)$, by the condition that $H_1(A, \ZZ)=W_\ZZ $ is a Hodge structure. Andr\'e has shown that $\psi_L$ can be defined over some number field, and this construction is the main ingredient to prove the polarized Shafarevich conjecture in \cite{A96}.

If one takes $L$ to be $\Lambda_h$,  the composition 
$$\psi_h\colon \widetilde{\scF}_{d,K}^{\dag}\to \Sh_K(\Lambda_h)\to \Sh_{K'}(\GSp(W),\Omega^{\pm})$$
is the classical Kuga--Satake morphism constructed by Deligne \cite{DeligneK3} and André \cite{A96}.

\subsection{Uniform Kuga--Satake for hyper-Kähler varieties}
\label{subsec:UniformKS}
Given a deformation type $M$ of hyper-Kähler varieties  with $\Lambda_M=\Lambda$, we can define the so-called \textit{uniform Kuga--Satake construction}, which extends the construction in \cite{Sh17}. 
The idea of the uniform Kuga--Satake construction has its origin in \cite{Ch16}, which uses the theory of the moduli space of stable sheaves instead of the Shimura variety, and is called Zarhin's trick for K3 surfaces.

The first step is to embed the lattice $\Lambda_h$ of signature $(2,n)$ primitively into a fixed unimodular lattice of signature $(2,*)$ for any $h$.

\begin{lemma}[Uniform lattice]
\label{lemma:UniformKS}
Given an even lattice $(\Lambda,q)$ of signature $(3,n)$, there exists an even unimodular lattice $\Sigma$ of signature $(2, *)$, depending only on $\rk (\Lambda)$, such that there is a primitive embedding 
\begin{equation}\label{uniks}
h^\perp=\Lambda_h\hookrightarrow \Sigma,
\end{equation}
for any $h\in \Lambda$ with $q(h)>0$. 
\end{lemma}
\begin{proof}
By Nikulin \cite[Theorem 1.12.4]{Ni79}, any even lattice $L$ of signature $(2,n)$ admits a primitive embedding into a fixed  even unimodular lattice $\Sigma$ of signature $(2,N)$ with $N\geq 2\rk (L)-2$ and $8|(N-2)$. One can thus choose $N\geq 2n+2$ satisfying $8 | (N-2)$ and let $$\Sigma= U^{\oplus 2}\oplus E_8^{\oplus d}$$ with $d=\frac{N-2}{8}$.
\end{proof}

\begin{remark}\label{odd}
If the Beauville--Bogomolov lattice $\Lambda$ is not even (cf. Remark \ref{rmk:evenness}),  Lemma \ref{lemma:UniformKS} does not apply. Nevertheless, we can always multiply the Beauville--Bogomolov quadratic form on $\rH^2(X,\ZZ)$ by $2$; this operation will not affect any construction. 
In the sequel, we assume that the lattice $\Lambda$ (hence also $\Lambda_h$) is even .
\end{remark}

For a lattice $L$, we write 
$G_L=\SO(L_{\QQ})$ for short. The inclusion \eqref{uniks} defines a map $D_{\Lambda_h} \rightarrow D_{\Sigma}$, which gives an embedding of Shimura data 
$$\iota_h : (G_{\Lambda_h}, D_{\Lambda_h}) \rightarrow (G_{\Sigma}, D_{\Sigma}).$$ 
Let $K \subseteq G_\Sigma(\mathbb{A}_f)$ be a compact open subgroup. Then for any compact open subgroup $K_{h}\subseteq G_{\Lambda_h}(\mathbb{A}_f)$ with $\iota_h (K_{h}) \subseteq K$, we have a finite and unramified map 
\begin{equation} \label{uKS1}
\Sh_{K_h} (G_{\Lambda_h}, D_{\Lambda_h}) \rightarrow \Sh_K(G_\Sigma, D_{\Sigma}).
\end{equation}
defined over $\QQ$. 

If $K$ is contained in the discriminant kernel $K_{\Sigma}$, then we get a section as \eqref{eq:sectionSOSpin} for
\[
\ad(\CC) \colon\Sh_{K^{\operatorname{sp}}}(\GSpin(\Sigma), \widetilde{D}_L)(\CC) \to \Sh_{K}(G_{\Sigma}, D_{L})(\CC) =\Sh_K(\Sigma)(\CC).
\]
In particular, we take the following level-$m$ structures for $m \geq 3$
\begin{equation}\label{eq:conven}
 \quad  \mathbf{K}'_m\coloneqq K_{\Sigma, m}^{\ad}\quad \text{and}\quad \mathbf{K}_m'' \coloneqq \left\{g \in \GSp(W)(\widehat{\ZZ})\big| g \equiv 1 \bmod m\right\},
\end{equation}
where $W= (\Cl^+(\Sigma_{\QQ}), \varphi_e)$ the symplectic space described in \S \ref{subsec:KS}.  For any $d \in \ZZ_{\geq 1}$, we have the following commutative diagram of stacks over $\QQ$
\begin{equation}\label{eq:bigdiag}
\begin{tikzcd}[column sep= large]
& & \Sh_{K^{\operatorname{sp}}_{\Sigma,m}}(\GSpin(\Sigma), \widetilde{D}_{\Sigma}) \ar[d,"\spin"] \ar[ld, shift left,"\ad"] \\
\widetilde{\scF}^{\dag}_{d,\mathbf{K}_{m}} \ar[d] \ar[r,"j_{\Sigma} \circ \mathcal{P}_{d,\mathbf{K}_m}"]& \Sh_{\mathbf{K}'_{m}}(\Sigma) \ar[d] \ar[ru, "\gamma_{E_m}",dashed] & \Sh_{\mathbf{K}''_m}(\GSp(W), \Omega^{\pm})\\
\widetilde{\scF}_{d}^{\dag} \ar[r,"j_{\Sigma} \circ \mathcal{P}_d"] & \Sh(\Sigma)
\end{tikzcd}
\end{equation}
where $\gamma_{E_m}$ is the descent of a chosen section $\gamma_{\CC}$ to the finite abelian extension $E_m$ of $\QQ$ corresponding to $\RR_{>0} \QQ^{\times} \mathbf{N}(K^{\operatorname{sp}}_{\Sigma,m}) \subseteq \mathbb{A}_{\QQ}^{\times}$ via the class field theory. We now obtain a map
\begin{equation}
\Psi_{d,\mathbf{K}_m,\CC}^{\KS}\colon \widetilde{\scF}_{d,\mathbf{K}_{m}}^\dag (\CC)\rightarrow  \Sh_{\mathbf{K}''_m}(\GSp(W), \Omega^{\pm})(\CC)
\end{equation}
as the composition, which is quasi-finite. It is called the \textbf{uniform Kuga--Satake map}. By the construction, we have the following descent theorem. 
\begin{theorem}\label{thm:descentKuga}
The uniform  Kuga--Satake map  $\Psi^{\KS}_{d,\mathbf{K}_m,\CC}$
descends to  map 
\[
\Psi^{\KS}_{d,\mathbf{K}_m}\colon \widetilde{\scF}^{\dag}_{d,\mathbf{K}_{m}} \to \Sh_{\mathbf{K}''_n}(\GSp(W), \Omega^{\pm})\]
over the number field $E_m$. The field $E_m$ is clearly independent of $d$ and $h$.   
\end{theorem}
\begin{proof}
The $\Psi_{d,\mathbf{K}_m,\CC}^{\KS}$ has a descent over $E_m$ as the period map $\widetilde{\scF}_{d, \mathbf{K}_{m}}^{\dag}(\CC) \to \Sh_{\mathbf{K}_{m}}(\Lambda_h)(\CC)$ has a descent over $\QQ$ by Proposition \ref{prop:descentperiod}. It is clear that $E_m$ is independent of $h$. 
\end{proof}

\begin{definition}  
If $f \colon T \rightarrow \widetilde{\scF}_{d,\mathbf{K}_m}^{\dag}$ is a map corresponding to a polarized oriented hyper-Kähler space with level structure $(\fX, \xi,\alpha)$ over $T$. We define the uniform Kuga--Satake abelian space $\mathcal{A}_T \rightarrow T$ associated with $f$ as the pullback of the universal polarized abelian scheme on $\Sh_{\mathbf{K}''_m}(\GSp(W), \Omega)$ under the composition 
\begin{equation}
T \rightarrow  \widetilde{\scF}_{d,\mathbf{K}_{m}}^{\dag}\rightarrow \Sh_{\mathbf{K}''_m}(\GSp(W), \Omega^{\pm}).
\end{equation}

\end{definition}
Suppose $f\colon \Spec F \rightarrow \Sh_{\mathbf{K}_{m}'}(\Sigma)$ comes from a $F$-point  of $\widetilde{\scF}_{d,\mathbf{K}_{m}}^{\dag}$, which represents  the tuple of polarized hyper-Kähler variety $(X, H, \alpha)$ with a $\mathbf{K}_{m}$-level structure. 
\begin{proposition}\label{propertyofks}
Let  $(A, L)$ be the associated uniform Kuga--Satake polarized abelian variety of $(X,H,\alpha)$. Then there is a Galois-equivariant lattice  embedding
$$\rP^2_{\et}(X_{\bar{F}},\widehat{\ZZ}(1))\cong f^\ast (\underline{\Sigma})\hookrightarrow \End_{\Cl^+(\Sigma)}\left(\rH^1_{\et}(A_{\bar{F}}, \widehat{\ZZ})\right) $$
such that $\Gal(\overline{F}/F)$ acts trivially on the orthogonal complement.
\end{proposition}
\begin{proof}
This was proved in \cite{Pe15} and \cite{Sh17}. 
\end{proof}

\begin{corollary}\label{cor:kugasatakegoodreduction}
Let $\p \in \Spec R$ be a prime ideal of height $1$.
If $X$ has essentially good reduction at $\p$,
then the associated Kuga--Satake variety $A_X$ admits good reduction at $\p$, i.e.\, there exists a smooth projective model over the localization $R_{\p}$ whose generic fiber is isomorphic to $A_{X}$.
\label{lemKSgood}
\end{corollary}
\begin{proof}
By the argument in \cite[Lemma 9.3.1]{A96} (see also \cite[Proposition 4.2.4]{Takamatsu2020c}),
it suffices to show that
$
\rH^{2}_{\et} (X_{\overline{k}}, \ZZ_{\ell})
$
is unramified at $\p$ for some odd prime number $\ell$, which is different from the residue characteristic of $R_{\p}$.  Here, we use the fact that $m$ is coprime to the residue characteristic of $R_{\p}$ and $m\geq 3$, so that we can apply the Raynaud semi-abelian reduction criterion.

We take a finite unramified extension $V$ of $\widehat{R_{\p}}$ and a smooth proper algebraic space $\mathcal{Y}$ over $V$ as in Definition \ref{essentially good}. Let $F$ be the fraction field of $V$.
By the smooth and proper base change theorem, $\rH^{2}_{\et}(\mathcal{Y}_{\bar{F}}, \ZZ_{\ell})$ is unramified at $\p$ as a $\Gal(\bar{F}/F)$-representation.
By the proof of Proposition \ref{prop:bir-obs}, it concludes the proof.

\end{proof}

\section{Finiteness of twists admitting essentially good reduction}
\label{sectionMatsusakaMumford}

\label{sectiontwists}
As an intermediate step towards the Shafarevich conjecture for polarized hyper-Kähler varieties, we prove, in Proposition \ref{Proposition:FinitenessTwists2} below, the finiteness of twists/models of polarized hyper-Kähler varieties admitting essentially good reductions.
We will make use of results in Appendix \ref{appendixMatsusakaMumford} on specialization of birational automorphisms (Matsusaka--Mumford) in the setting of algebraic spaces.

\begin{lemma}
\label{ruled}
Let $R$ be a discrete valuation ring with fraction field $K$ of characteristic zero.
Let $s$ be the closed point of $\Spec R$ and $k(s)$ the residue field.
Let $\mathcal{X} \rightarrow \Spec R$ be a smooth proper morphism of algebraic spaces such that the generic fiber $\mathcal{X}_{\eta}$ is a hyper-Kähler variety over $K$.
Then the special fiber $\mathcal{X}_{s}$ is a non-ruled algebraic space over $k(s)$ (Definition \ref{def:Ruled}).
\end{lemma}

\begin{proof} 
The relative canonical sheaf
\[
\omega_{\mathcal{X}/R} \coloneqq \bigwedge \Omega_{\mathcal{X}/R}^{1}
\]
is trivial, since the special fiber $\mathcal{X}_{s}$ is a principal Cartier divisor of $\mathcal{X}$.
Therefore, the special fiber $\mathcal{X}_{s}$ is a smooth proper algebraic space over $k(s)$ with a trivial canonical sheaf.
Suppose, by contradiction, that $\mathcal{X}_{s}$ were ruled.
Take a normal projective variety $V$ over $k(s)$ such that $\PP^{1}_{k(s)} \times V$ is birationally equivalent to $\mathcal{X}_{s}$.
By using Chow's lemma and taking the elimination of indeterminacies,
there exists a projective normal scheme $W$ over $k(s)$ with projective birational morphisms $f\colon W \rightarrow \mathcal{X}_{s}$ and $g\colon W \rightarrow \PP^{1}_{k(s)} \times V$.
Take a canonical (Weil) divisor $K_{W}$ on $W$.
We define a Weil divisor $K_{X_{s}}$ as the push-forward $f_{\ast} (K_{W})$.
We also set $K_{\PP^{1}_{k(s)} \times V} \coloneqq g _{\ast} (K_{W})$. 
By étale descent, we have $\mathcal{O}(K_{X_{s}}) \simeq \omega_{\mathcal{X}_{s}/k(s)}$.
Moreover, we have a natural isomorphism $f_{\ast}(\mathcal{O}_{W}(K_{W})) \simeq \mathcal{O}_{X_{s}}(K_{X_{s}})$, since $X_{s}$ has terminal singularities étale locally.
On the other hand, we have a natural inclusion
$g_{\ast} (\mathcal{O}_{W}(K_{W})) \hookrightarrow \mathcal{O}_{\PP^{1}_{k(s)} \times V} (K_{\PP_{1} \times V})$.
Thus $\Gamma(\PP^{1}_{k(s) \times V},\mathcal{O}_{\PP^{1}_{k(s)} \times V}(K_{\PP^{1}_{k(s)} \times V})) \neq 0$, a contradiction.
\end{proof}

\begin{lemma}
\label{Lemma:Unramifiedness}
Let $R$ be a henselian discrete valuation ring with $K = \Frac R$.
Let $(X, \mathscr{L})$, $(Y,\mathscr{M})$ be polarized hyper-Kähler varieties over $K$. Let
$f \colon (X_{\overline{K}}, \mathscr{L}_{\overline{K}})
\overset{\simeq}{\to}
(Y_{\overline{K}}, \mathscr{M}_{\overline{K}})$
be an isomorphism of polarized varieties over $\overline{K}$
with the associated 1-cocycle
\[ \alpha_{f} \colon \Gal(\overline{K}/K) \to \mathscr{G} \coloneqq \Aut(X_{\overline{K}}, \mathscr{L}_{\overline{K}}) \]
defined by %
$\alpha_{f} (\sigma) = f^{-1} \circ ^{\sigma} \!f$.
Assume that
\begin{itemize}
\item the order of $\mathscr{G}$ is invertible in the residue field of $R$,
\item there exist proper smooth algebraic spaces
$\mathcal{X}', \mathcal{Y}'$ over $R$
whose generic fibers are hyper-Kähler varieties
such that there are birational maps
$X \dasharrow \mathcal{X}'_{\eta}$,
$Y \dasharrow \mathcal{Y}'_{\eta}$ over $K$.
\end{itemize}
Then the restriction of $\alpha_{f}$ to the inertia subgroup
$\mathrm{I}_K \subset \Gal(\overline{K}/K)$
is trivial.
\end{lemma}

\begin{proof}
Take a finite Galois extension $L/K$ such that
$\mathscr{G} = \Aut(X_{L}, \mathscr{L}_{L})$ and 
$f$ come from an isomorphism
$f_L \colon (X_L, \mathscr{L}_{L})
\overset{\simeq}{\to}
(Y_L, \mathscr{M}_{L})$ over $L$.
The 1-cocycle  $\alpha_f$ comes from a $1$-cocycle
$$\alpha_{f,L} \colon \Gal(L/K) \to \mathscr{G}$$
corresponding to $f_L$.
Let $\rI_{L/K} \subset \Gal(L/K)$ be the inertia group.
For an element $\sigma \in \rI_{L/K}$, we have
$\alpha_{f,L}(\sigma) \coloneqq (f_L)^{-1} \circ ^{\sigma} \!f_L$ 
and 
we shall show that $\alpha_{f,L}(\sigma) = \id$.

Let $\mathcal{X}'_{\mathscr{O}_L,s}$, $\mathcal{Y}'_{\mathscr{O}_L,s}$
be the special fibers of
$\mathcal{X}'_{\mathscr{O}_L} \coloneqq \mathcal{X}' \otimes_{\mathscr{O}_K} \mathscr{O}_L$,
$\mathcal{Y}'_{\mathscr{O}_L} \coloneqq \mathcal{Y}' \otimes_{\mathscr{O}_K} \mathscr{O}_L$,
respectively.
We note that the isomorphism $f_L \colon X_L \overset{\simeq}{\to} Y_L$
can be considered as a birational map
\[f_L \colon \mathcal{X}'_{L} \overset{\simeq}{\dashrightarrow} \mathcal{Y}'_{L}.\]
Here, we put $\mathcal{X}'_{L} \coloneqq \mathcal{X}' \otimes_{\mathscr{O}_K} L$ and
$\mathcal{Y}'_{L} \coloneqq \mathcal{Y}' \otimes_{\mathscr{O}_K} L$. By Theorem \ref{ClosedSubset} and Lemma \ref{ruled},
the birational map $f_L$
extends to an isomorphism
\[
  \widetilde{f}_L \colon
  (\mathcal{X}'_{\mathscr{O}_L} \setminus V) \overset{\simeq}{\to}
  (\mathcal{Y}'_{\mathscr{O}_L} \setminus W)
\]
over $\mathscr{O}_L$
for some proper closed subsets
$V \subset \mathcal{X}'_{\mathscr{O}_L}$, $W \subset \mathcal{Y}'_{\mathscr{O}_L}$
satisfying
$V_s \neq \mathcal{X}'_{\mathscr{O}_L,s}$, $W_s \neq \mathcal{Y}'_{\mathscr{O}_L,s}$.
Enlarging $V$ and $W$ if necessary, by Lemma \ref{Lemma:Specialization},
we may assume the following holds:
The action of
$\mathscr{G} = \Aut(X_{L}, \mathscr{L}_{L})$ on $X_L$
extends to a homomorphism
\[
  \psi \colon \mathscr{G} \to \Aut(\mathcal{X}'_{\mathscr{O}_L} \setminus V).
\]

Since $\sigma \in \rI_{L/K}$ acts trivially on the residue field of $\mathscr{O}_L$,
the isomorphisms
\[
  \widetilde{f}_L \colon
  (\mathcal{X}'_{\mathscr{O}_L} \setminus V) \overset{\simeq}{\to}
  (\mathcal{Y}'_{\mathscr{O}_L} \setminus W)
  \quad \text{and} \quad
  ^{\sigma} \!\widetilde{f}_L \colon
  (\mathcal{X}'_{\mathscr{O}_L} \setminus \sigma(V)) \overset{\simeq}{\to}
  (\mathcal{Y}'_{\mathscr{O}_L} \setminus \sigma(W))
\]
induce the same morphism on the special fiber.
Thus, the element
\[
 \alpha_{f,L}(\sigma) = (f_L)^{-1} \circ ^{\sigma}\!f_L \in \mathscr{G} \subset \Bir (\mathcal{X}_{L})
 \]
sits in the kernel of the specialization map
\[
  \overline{\psi} \colon \mathscr{G} \to \Aut((\mathcal{X}'_{\mathscr{O}_L} \setminus V)_s).
\]
Since the order of $\mathscr{G}$ is invertible
in the residue field of $R$,
the specialization map
$$\mathscr{G} \to \Aut((\mathcal{X}'_{\mathscr{O}_{L}} \setminus V)_{s})$$
is injective by Lemma \ref{Lemma:Specialization} (2).
Therefore,  we have $\alpha_{f,L}(\sigma) = \id$.
\end{proof}

\begin{proposition}
\label{Proposition:FinitenessTwists2}
Let $R$ be a finitely generated $\ZZ$-algebra that is a normal integral domain with fraction field $F$.
Let $X$ be a hyper-Kähler variety over $F$ which admits essentially good reduction at any height $1$ prime ideal $\p \in \Spec R$.
Let $\mathscr{L}$ be an ample line bundle on $X$.
Then there exist only finitely many isomorphism classes of
polarized varieties $(Y,\mathscr{M})$ over $F$
satisfying the following conditions$\colon$
\begin{itemize}
\item 
$Y$ admits essentially good reduction at any height $1$ prime $\p \in \Spec R$.
\item There exists an isomorphism of polarized varieties
$(X_{\overline{F}},\mathscr{L}_{\overline{F}}) \simeq (Y_{\overline{F}},\mathscr{M}_{\overline{F}})$
over $\overline{F}$.
\end{itemize}
\end{proposition}

\begin{proof}
We put $\mathscr{G} \coloneqq \Aut(X_{\overline{F}}, \mathscr{L}_{\overline{F}})$.
Shrinking $\Spec R$ if necessary, we may assume that the order of $\mathscr{G}$ is invertible
in $R$.
Note that $\mathscr{G}$ can be considered as a finite subgroup of
$\Bir(X_{\overline{F}})$.

We take a finite Galois extension $L/F$ such that $\Aut(X_{L},\mathscr{L}_{L}) = \Aut (X_{\overline{F}}, \mathscr{L}_{\overline{F}})$.
Let $$f \colon (X_{\overline{F}}, \mathscr{L}_{\overline{F}})
\simeq
(Y_{\overline{F}}, \mathscr{M}_{\overline{F}})$$
be any isomorphism of polarized varieties over $\overline{F}$, where $(Y, \mathscr{M})$ is any polarized variety satisfying the properties in the statement.
The associated $1$-cocycle
$\alpha_{f} \colon \Gal(\overline{F}/F) \to \mathscr{G}$
is defined by
$\alpha_{f}(\sigma) \coloneqq (f)^{-1} \circ ^{\sigma}\!f$.
It is enough to show that
there are only finitely many possibilities for the $1$-cocycle $\alpha_f$.

Note that $\alpha_{f}|_{\Gal(\overline{F}/L)} \colon \Gal(\overline{F}/L) \rightarrow \mathscr{G}$ is a group homomorphism.
Therefore, the field extension $M/L$ corresponding to the kernel of $\alpha_{f}|_{\Gal(\overline{F}/L)}$ satisfies
$[M:L] < |\mathscr{G}|$.
Let $R'$ be a normalization of $R$ in $L$, and $\p$ a height $1$ prime ideal of $R'$.
By the assumption, there exists a finite \'etale extension $\widehat{R'}_{\p} \subset S$ such that
there exist smooth proper algebraic spaces over $S$ whose generic fibers are hyper-Kähler varieties that are birationally equivalent to $X_{\Frac S}$ and $Y_{\Frac S}$, respectively.
Therefore, by Lemma \ref{Lemma:Unramifiedness}, the inertia subgroup $\rI_{\p} \subset \Gal(\overline{F}/L)$ (which is defined by fixing the extension of valuation $\p$ to $\overline{F}$) is contained in the kernel of $\alpha_{f}|_{\Gal(\overline{F}/L)}$.
By Zariski--Nagata's purity theorem, $\alpha_{f}|_{\Gal(\overline{F}/L)}$ factors through $\pi_{1}(\Spec R', \overline{F})$.
By \cite[Proposition 2.3, Theorem 2.9]{Harada2009}, the family of subsets
\[
\mathcal{C}\coloneqq \{ H \subset \pi_{1}(\Spec R', \overline{F})\colon \text{open subgroup} \mid [\pi_{1}(\Spec R', \overline{F}) : H] \leq |\mathscr{G}|\}
\]
is a finite set.
Let $M'$ be the fraction field of the finite \'etale cover of $R'$ corresponding to $\bigcap_{H \subset \mathcal{C}} H$, and $M''$ the Galois closure of $M'$ over $K$. We note that $M''$ does not depend on $(Y,\mathscr{M})$. 
Then $\alpha_{f}$ factors through the finite group $\Gal(M''/L)$, so there are only finitely many possibilities for the $1$-cocycle $\alpha_f$. This finishes the proof.
\end{proof}

\section{Finiteness of (unpolarized) Shafarevich sets} 
\label{sec:FinitnessResults}
In this section, we prove the unpolarized Shafarevich conjecture stated in \Cref{mainthm}.

\subsection{Some reductions}In order to apply the uniform Kuga--Satake map, we need to associate each polarized oriented hyper-Kähler variety with a level structure. We start with the following result.

\begin{lemma}Keep the notations the same as in \eqref{eq:conven}. Let $M$ be a $\widehat{\ZZ}$-numerically equivalent class of hyper-Kähler varieties. For a hyper-Kähler variety $X$ over a field $F$ of characteristic $0$ in $M$, we have
\[
[K_{\Lambda_h}: \mathbf{K}_m] \leq C m^{2^{(b_{2}(M)-2)}},
\]
where we denote
\[
C\coloneqq [\SO(H^{2}_{\et} (X_{\overline{F}},\ZZ_{2})): f (\GSpin (H^{2}_{\et} (X_{\overline{F}}. \ZZ_{2})))],
\]
and $f\colon \GSpin \rightarrow \SO$ is the natural homomorphism defined by the conjugation.
\label{lemlevel}
\end{lemma}
\begin{proof}
This follows from the same argument as in \cite[Corollary 3.1.8]{Takamatsu2020K3}.
\end{proof}

Let $\delta(m,M)$ be the constant $C m^{2^{(b_{2}(X_{\overline{F}})-2)}}$ as in Lemma \ref{lemlevel}, which depends only on $m$ and the $\widehat{\ZZ}$-numerical type $M$.

\begin{proposition}
\label{unplevel}
Let $R$ be a finitely generated $\ZZ$-algebra that is an integral domain, and let $F$ be its fraction field. Let $M$ be a $\widehat{\ZZ}$-numerical type.
\begin{enumerate}
    \item[(i)] For any oriented polarized hyper-Kähler variety $(X,H,\omega)$ of type $M$ defined over $F$, there is a finite field extension $E/F$ of degree $\leq \delta(m,M)$, such that $(X_E,H_E)$ is equipped with a $\mathbf{K}_m$-level structure.
    \item[(ii)] There exists a finite Galois extension $E_m$ over $F$ such that for every element $X \in \Sha_{M}^{\ess}(F,R)$, the following hold.
\begin{enumerate}
\item
there exists a polarization $\lambda$ on $X$ of degree $d$, an orientation $\omega$ on $X_{E_m}$, a $\mathbf{K}_m$-level structure $\alpha$ on $X_{E_m}$.
\item
$\Pic_{X_{E_m}} = \Pic_{X_{\overline{F}}}$.
\end{enumerate}
\end{enumerate}
\end{proposition}

\begin{proof}
For $(i)$, we note that we have the Cartesian diagram \eqref{eq:cartisianlevel}.
It suffices to show that  for any  map $\Spec F\to \Sh(\Lambda_h)$ there exists $E/F$ such that we can find a lift to $\Sh_{\mathbf{K}_{m}}(\Lambda_h)$ after a finite base change to $E$. The image of $\Gal(\bar{F}/F)$ lies in $\mathbf{K}_m$ if the composition of continuous group homomorphisms 
\[
\Gal(\bar{F}/F) \xrightarrow{\rho_x} K_{\Lambda_h} \to K_{\Lambda_h}/\mathbf{K}_m
\]
is trivial.  Thus we can take $E$ to be the Galois extension corresponding to the kernel of this composition. We can see $[E:F] \leq \delta(m,M)$ by Lemma \ref{lemlevel}.

For $(ii)$, we notice that the representation $\rho_x$ factors through a continuous homomorphism $\rho_{x,R} \colon \pi_1(R, \bar{F}) \to K_{\Lambda_h}$ under the condition of essentially good reduction. The Hermite–Minkowski theorem (see \cite[Chap. VI, Sect. 2]{Faltings1992} or \cite[Proposition 2.3, Theorem 2.9]{Harada2009}) implies that there are only finitely many compositions of $\rho_{x,R}$ with the quotient map $K_{\Lambda_h} \to K_{\Lambda_h}/\mathbf{K}_m$. Thus we can take a finite dominant \'etale morphism $\Spec(\widetilde{R}) \to \Spec(R)$, so that the image of $\pi_1(\widetilde{R},\bar{F}) \subset \pi_1(R,\bar{F})$ lies in $\mathbf{K}_m$ under any such $\rho_x$.
Let $E$ be the fraction field of $\widetilde{R}$. Then any $X_E$ will be equipped with a $\mathbf{K}_m$-level structure by the construction. With a similar argument, we can also assume that any $X_E$ has an orientation $\omega$.

On the other hand,
the image of the Galois action
\[
\gamma \colon G_{F} \rightarrow \GL (\Pic_{X_{\overline{F}}})
\]
is a finite subgroup.
Since a finite subgroup is mapped injectively to $\GL(\Pic_{X_{\overline{F}}}\otimes_{\ZZ}\mathbb{F}_{3})$, we have $|G_{F}/\ker \gamma| \leq |\GL(r,\mathbb{F}_{3})|$, where $r$ is the Picard number of $X_{\overline{F}}$, which is bounded above by $N$.
Therefore, we can  take a field $E_m$ satisfying both $(a)$ and $(b)$ by the Hermite–Minkowski theorem.
\end{proof}

The following result allows us to take finite field extensions in proving the finiteness of Shafarevich sets. 

\begin{lemma}\label{red-fin-ext}
Let $R_E/R$ be a finite Galois extension between finitely generated normal domains whose fraction field is $E/F$.  We define 
	\begin{equation*}
	\Sha_{M}(E/F,R_E)=\left\{  X'\in \Sha_M (E,R_E)~|X'\cong X\times_F \Spec(E) ~\hbox{for some}~ X \in \Sha_M(F, R) \right\}
	\end{equation*}
For any $X \in  \Sha(F, R)$, the set
	\begin{equation*}	
	\{Y\in \Sha_M(F,R)|~Y_E \cong X_E\}
	\end{equation*}
	is finite. 	If $\Sha_{M}(E/F,R_E)$ is finite, then $\Sha_M(F,R)$ is finite. Similar result holds for the finiteness of birational isomorphism classes of the Shafarevich sets.
\end{lemma}
\begin{proof}
	See \cite[Lemma 4.1.4]{Sh17}.
\end{proof}

\subsection{Finiteness of Picard lattices} The key application of the uniform Kuga--Satake map is the following finiteness result on Picard lattices; compare to She \cite[Corollary 4.1.14]{Sh17}.
\begin{theorem}\label{boundedpolarization}
 Suppose $F$ contains the number field $E_{4}$ in Proposition \ref{unplevel} (ii). Then the set \[
\{\Pic_{X}  \mid  X \in \Sha_M^{\ess}(F,R)
\}/ \text{lattice isometry}
\]
is finite.
\end{theorem}
\begin{proof}
 This follows from the same argument as in \cite{Sh17,Takamatsu2020K3}. We nevertheless give some details of the proof for  hyper-Kähler varieties. The main difference to the case of K3 surfaces is that the BB-form is now not unimodular in general.  As there are only finitely many deformation types with a given $\widehat{\ZZ}$-numerical type, we may assume that all $X\in \Sha_M^{\ess}(F,R)$ are geometrically deformation equivalent. Let $E$ be the finite extension of $F$ such that all elements in $\Sha_M^{\ess}(F,R)$ admit $\mathbf{K}_4$-level structures.
Consider the transcendental lattice 
\[
\rT(X_E)\coloneqq\Pic_{X_E}^\perp \subseteq \rH^2(X_\CC,\ZZ).
\]
Then from the Tate conjecture over finitely generated fields (\Cref{prop:TateConjecture}),  we know that  \[
\rT(X_E)\otimes \widehat{\ZZ}\cong (\rP^2_{\et} (X_{\bar{E}},\widehat{\ZZ}(1))^{G_E})^\perp \subseteq \rP^2_{\et}(X_{\bar{E}},\widehat{\ZZ}(1)).
\]

On the other hand, we have 
\[
(\rP^{2}_{\et} ((X_{\overline{F}}, \lambda_{\overline{F}}), \widehat{\ZZ}(1))^{G_{F}})^{\perp} = (f^*(\underline{\Sigma})^{G_{F}})^{\perp},
\]
by Proposition \ref{propertyofks}. Since the image in $\Sh_{\mathbf{K}_4'}(\Sigma)(E)$ of the $E$-isomorphism classes in the set $\Sha^{\hom_\ell}_M(F,R)$ is finite by Lemma \ref{lemKSgood}, \cite[VI, \S1, Theorem 2]{Faltings1992}, we can show that
\[
\{
(f^*(\underline{\Sigma})^{G_{F}})^{\perp} \mid X \in \Sha_M^{\ess}(F,R)
\}/\text{lattice isometry}
\]
is a finite set.

Now, using \cite[San 30.2]{Kn02}, we know that  the collection of  primitive embeddings $$\rT(X_E)\rightarrow \Lambda$$ is finite.  It follows that the set of  lattices $\Pic_{X_E}=\rT(X_E)^\perp$  is finite. As there are only finitely many conjugacy classes of homomorphisms $\Gal(E/F) \to \mathrm{O}(\Pic_{X_E})$ (see \cite[Theorem 4.3]{PR94}), we get the finiteness of $\Pic_X$.
\end{proof}

\subsection{Proof of Theorem \ref{mainthm}}
 Let us consider the natural maps
\begin{equation}\label{eq:essentialmaps}
\coprod_{d \leq N} \Sha^{\ess}_{M,d}(F,R) \to \Sha_M^{\ess}(F,R) \twoheadrightarrow \Sha_M^{\ess}(F,R)_{\slash \sim F\!\bir}.
\end{equation}

First, we have the following result for the polarized Shafarevich set with essentially good reductions.
\begin{theorem}\label{thmpol}
$\Sha_{M,d}^{\ess}(F,R)$ is a finite set.
\end{theorem}
\begin{proof}
Thanks to Proposition  \ref{Proposition:FinitenessTwists2},  it suffices to show there are only finitely many geometric isomorphism classes in  $\Sha^{\ess}_{M,d}(F,R)$.  In \cite[Section 9.4]{A96}, Andr\'e has  shown that  there are only finitely many geometric isomorphism classes  in $\Sha_{M,d}(F,R)$ via the usual Kuga--Satake construction. A similar proof also applies here:

For any $(X,H) \in \Sha^{\ess}_{M,d}(F,R)$ equipped with a $\mathbf{K}_4$-level structure $\alpha$ over $F$, the corresponding uniform Kuga--Satake abelian variety has good reductions  according to Corollary \ref{cor:kugasatakegoodreduction}. Therefore, the image of the set of $\mathbb{C}$-points corresponding to those $(X,H,\alpha)$, under $\Psi_{d,\mathbf{K}_4,\CC}^{\KS}$ is finite by \cite[VI, \S1, Theorem 2]{Faltings1992}. As the uniform Kuga--Satake map is geometrically quasi-finite, this means that the geometric isomorphism classes of $(X,H,\alpha)$ are finite, which proves our assertion. 
\end{proof}

\begin{theorem}\label{finbirsh} 
There exists an integer $N$ such that for any $X\in \Sha_{M}^{\ess}(F,R)$,  there is a polarized hyper-Kähler variety $(Y,H)\in \Sha_{M,d}^{\ess}(F,R) $ such that $Y$ is $F$-birational to $X$ and $d\leq N$. In other words, the composition map in \eqref{eq:essentialmaps} is surjective for some $N$.
\end{theorem}
\begin{proof}
By Theorem \ref{boundedpolarization}, there are only finitely many isomorphism classes of lattices $\Pic_X$  for $X\in \Sha_{M}^{\ess}(F,R)$.  Fix an embedding $F\hookrightarrow \CC$.  By  Proposition \ref{unplevel} (ii) and Lemma \ref{red-fin-ext},  we may further assume that $\Pic_X\cong \Pic(X_{\CC})$ for every $X\in \Sha_M^{\ess}(F,R)$. 
Let us fix a (geometric) Picard lattice $\Xi$ and consider the moduli space of  $\Xi$-lattice polarized hyper-Kähler varieties of a given deformation type $M$. As it is of finite type over a field, it suffices to consider its geometric irreducible components one by one.

Fix such a component and take an $X\in \Sha_{M}^{\ess}(F,R)$ in this component such that $\Pic(X_\CC)\cong \Pic_X \cong\Xi$ (if there is no such $X$, then the component is irrelevant to the question). 
We claim that there exists an integer $N$, such that for each hyper-Kähler variety $X'$ defined over $F$ satisfying the following conditions:
\begin{itemize}
    \item $\Pic(X'_\CC)\cong \Pic_{X'}\cong \Xi$,
    \item the $\Xi$-lattice-polarized hyper-Kähler manifolds $(X_\CC, \Pic(X_\CC)$ and $(X'_\CC, \Pic(X'_\CC))$ are deformation equivalent; 
\end{itemize}
there exists a hyper-Kähler variety  $Y$ defined over $F$ which is $F$-birational to $X'$ and endowed with a polarization $H$ of Beauville--Bogomolov square $\leq N$. 

Let $\mathcal{W}(X_\CC)\subseteq \Xi$ be the collection of wall divisors on $X_\CC$.
We define $N$ as follows:
\begin{equation}\label{bound-ba}
    N\coloneqq\inf\{ v^2>0|~v\in\Xi, v\cdot w\neq 0,~\forall w\in \mathcal{W}(X_\CC)\}.
\end{equation}
From the assumption,  the deformation invariance of wall divisors (see \cite{AV15})  induces an identification between $\mathcal{W}(X'_\CC)$ and $\mathcal{W}(X_\CC)$. Recall that the set of birational ample classes in $\Pic_{X'}$ is given by  
$$\{ v\in \Pic_{X'}|~v\cdot w\neq 0~\forall w\in \mathcal{W}(X'_\CC)\}.$$
It follows that  $d\coloneqq\inf \{ v^2>0|~v\in \Pic_{X'}~\hbox{and}~v\cdot w\neq 0~\forall w\in \mathcal{W}(X'_\CC)\}$ is no greater than $N$.  Take $H'\in \Pic_{X'}$ with $(H')^2=d$, then consider the polarized variety over $F$:
\[(Y, H)\coloneqq\left(\operatorname{Proj}\bigoplus_m H^0(X', \mathcal{O}(mH')), \mathcal{O}(1)\right).\]
By the minimal model theory for hyper-Kähler varieties, $(Y_\CC,H)$ is a polarized hyper-Kähler variety, which is $F$-birationally equivalent to $X$, hence of deformation type $M$ by \cite[Theorem 4.6]{MR1664696}. Hence $(Y, H)$ is also hyper-Kähler and satisfies all the desired properties. 
\end{proof}

Now let us conclude the proof of Theorem \ref{mainthm}. The finiteness of $\Sha_M^{\ess}(F,R)_{/ F\!\bir}$  follows  from  the combination of Theorem \ref{thmpol} and Theorem \ref{finbirsh}. If $b_2(M) \geq 5$, then the second map in \eqref{eq:essentialmaps} has finite fibers by Theorem \ref{cor:bir-fin}, and the finiteness of $\Sha_M^{\ess}(F,R)$ follows. \qed

\section{Cohomological Shafarevich conjecture}
\label{sectionremcoh}

In this section, we give examples where $\Sha^{\hom}_M(F,S)$, even $\Sha^{\hom}_{M,d}(F,S)$ is infinite. Hence, the naive cohomological generalization of the Shafarevich conjecture fails in general. We propose a remedy by taking into account the cohomology of degrees other than 2.

\subsection{Failure of the (second) cohomological Shafarevich conjecture}
\label{subsec:Fail-secondCohomShafConj}
Let $X$ be a hyper-Kähler variety defined over a number field $F$ and $H$ is a polarization of degree $d$.  There is a bijection between the set (assuming nonempty) of all $F$-forms of $X_{\bar{F}}$  and the Galois cohomology 
\begin{equation}\label{eq:F-galois}
    \rH^1(\Gal(\bar{F}/F), \Aut(X_{\bar{F}})),
\end{equation} 
via the construction of twisting \cite{BorelSerre}. Since $\Aut(X_{\bar{F}})$ is finitely generated by Cattaneo--Fu \cite{CattaneoFu}, up to replacing $F$ by a finite extension, we can assume that $\Gal(\bar{F}/F)$ acts trivially on $\Aut(X_{\bar{F}})$. 

Recall from Section  \ref{subsect:AutBir} that
\[
\Aut_0(X)\coloneqq\ker\left(\Aut (X) \rightarrow \GL (\rH^{2}_{\et}(X_{\overline{F}},\widehat{\ZZ}))\right).
\]
As a result, $\rH^1(\Gal(\bar{F}/F), \Aut_0(X_{\bar{F}}))$ is identified with a subset of \eqref{eq:F-galois}. If $X$ is in $\Sha^{\hom}_{M,d}(F,S)$, then twisting by distinct elements in $\rH^1(\Gal(\bar{F}/F), \Aut_0(X_{\bar{F}}))$ gives rise to different $F$-forms of $(X_{\bar{F}}, H_{\bar{F}})$ (resp. $X_{\bar{F}}$), while the unramifiedness condition on the second cohomology is preserved by construction, i.e. all those $F$-forms are in $\Sha^{\hom}_{M,d}(F,S)$ (resp. $\Sha^{\hom}_{M}(F,S)$). If  $\Aut_0(X_{\overline{F}})$ is nontrivial, the cohomology $$\rH^1(\Gal(\bar{F}/F), \Aut_0(X_{\bar{F}}))\simeq \operatorname{Hom}(\Gal(\bar{F}/F),\Aut_0(X_{\bar{F}}))$$ is infinite in general, giving rise to infinitely many elements in $\Sha^{\hom}_{M,d}(F,S)$.

Let us give some concrete examples. 
Let $X$ be a $2n$-dimensional generalized Kummer type hyper-Kähler variety. As before, by using \cite{CattaneoFu}, we may assume that $\Aut(X_{\overline{F}}) = \Aut (X)$. By \cite[Corollary 5]{BNS11}, $\Aut_0(X)$
is isomorphic to the semi-direct product $(\ZZ/(n+1)\ZZ)^{4} \rtimes (\ZZ/2\ZZ)$. As there are infinitely many degree 2 extensions of $F$, we obtain infinitely many homomorphisms 
\[\Gal(\overline{F}/F) \twoheadrightarrow \ZZ/2\ZZ,\]
hence infinitely many elements in $\operatorname{Hom}(\Gal(\overline{F}/F), \Aut_0(X))$. Similarly, for $OG_6$-type hyper-Kähler varieties $X$, $\Aut_0(X)$ is isomorphic to $(\ZZ/2\ZZ)^{\times 8}$. By \cite[Theorem 5.2]{MW17}, we obtain infinitely many elements in $\Sha^{\hom}_{M,d}(F,S)$.

We have obtained the following result:

\begin{proposition}
 \label{propgenkum}
    Let $\ell$ be any prime number, and $S$ the set of places consisting of places $v$ with $v\mid \ell$ and ramified places of the $\Gal(\overline{F}/F)$-module $\rH^2_{\et}(X_{\overline{F}},\ZZ_{\ell})$.
    \begin{enumerate}
    \item 
    The following subset of $\Sha^{\hom_\ell}_{M,d}(F,S)$ is infinite:
    \begin{equation}
        \label{eqn:InfiniteTwists}
        \Set*{Y \in \Sha^{\hom_\ell}_{M,d}(F,S)   \given 
        \begin{array}{l}
            Y_{\overline{F}}\cong X_{\overline{F}} 
        \end{array}}/{\cong_F}
    \end{equation}
\item
Let $T$ be any finite set of finite places of $F$.
Then for almost all $X$ in \eqref{eqn:InfiniteTwists} , 
there exists a finite place $v \notin S \cup T$ (depending on $X$) such that
$X$ does not admit essentially good reduction at $v$ though $\rH^{2}_{\et}(X_{\overline{F}},\ZZ_{\ell})$ is unramified at $v$. 
Moreover, enlarging $T$ if necessary, we may assume that $X$ admits a potentially good reduction at $v$.
\end{enumerate}
\end{proposition}
The (2) above states that the analogue of \cite[Theorem 1.3 (ii)]{LM18} (see also \cite[Introduction]{Takamatsu2020K3}) does not hold for higher-dimensional hyper-Kähler varieties when $\Aut_0$ is nontrivial.

\subsection{Full degree cohomological generalization} 
\label{subsec:FullDegCohomShafConj}
We have seen that the non-faithfulness of the action of automorphisms on $\rH^2$ leads to the infiniteness of the cohomological Shafarevich sets. We provide a remedy by taking into account more cohomological degrees other than 2.  Throughout this section, let $I$ be a subset of $\NN$ containing 2. 

Define the following group
\begin{equation}
    \label{def:AutI}
\Aut_I(X)\coloneqq \bigcap_{i\in I}\ker\left(\Aut (X_{\overline{k}}) \rightarrow \GL\left(\rH^{i}_{\et}(X_{\overline{k}},\QQ_{\ell}) \right)\right). 
\end{equation}
By the Artin comparison theorem, the group $\Aut_I(X)$ is independent of $\ell$, and $\Aut_0(X)$  (resp.~$\Aut_{00}(X)$) defined in Section \ref{subsect:AutBir}
are nothing but $\Aut_{\{2\}}(X)$ (resp.~$\Aut_{\NN}(X)$).

We use the following definition of level structure on the cohomology groups in several degrees. Let $k$ be a field of characteristic $0$, and $X$ a hyper-Kähler variety over $k$.
We denote the torsion-free part of $\bigoplus_{i\in I}\rH^i_{\et}(X_{\overline{k}},\ZZ_{\ell})$ by $\rH^I_{\ell}(X_{\overline{k}})$.
\begin{definition}
Let $\rH$ be the abstract $\ZZ_{\ell}$-module which is isomorphic to $\rH^I_{\ell}(X_{\overline{k}})$.
A \emph{level $\ell^{n}$ structure} $\beta$ on $X$ of degree $I$ is a $\Gal(\overline{k}/k)$-invariant
$\GL(\rH,\ell^{n})$-orbit of an isomorphism of $\ZZ_{\ell}$-modules
\[
\rH \xrightarrow{\sim} \rH^I_{\ell}(X_{\overline{k}}).
\]
Here, we put
\[
\GL(\rH,\ell^{n}) \coloneqq\{g \in \GL (\rH) \mid g \equiv 1 \textup{ mod } \ell^{n} \}.
\]
\end{definition}

Now let $R$ be a finitely generated $\ZZ$-algebra which is a normal domain, with fraction field denoted by $F$. Suppose that $1/\ell \in R$.
Define the following \textit{generalized} cohomological Shafarevich set:
\begin{equation}
    \Sha_{M}^{I\hbox{-}\hom_\ell}(F,R)\coloneqq \left\{
  X\in \Sha_M^{\hom_\ell}(F,R) \left|
  \begin{array}{l}
    \rH^{I}_{\ell}(X_{\overline{F}}) \text{is unramified }\\
  \text{at any height 1 prime } \p\in \Spec R \\
  \end{array}
  \right.
\right\}/{\cong_F}.
\end{equation}

\begin{theorem}
\label{cohshaf}
Consider $M$ to be a deformation type or a $\widehat{\ZZ}$-numerical type of hyper-Kähler varieties with $b_2(M)\geq 5$. The set
\begin{equation}
\Set*{
  X\in \Sha_M^{I\hbox{-}\hom_\ell}(F,R) \given 
    \Aut_{I}(X)=\{1\}
}/{\cong_F}
\end{equation}
is finite.
\end{theorem}

\begin{proof} 
The proof is similar to \cite[Section 4]{Takamatsu2020K3}. We need to use integral models of Shimura varieties, an idea taught to the third author by Yoichi Mieda. 

  Assume $b_2(M)\geq 4$. We will first prove the polarized case, i.e.\,the finiteness of the following set for any positive number $d$:
\begin{equation}
S_1\coloneqq \left\{
  (X, \lambda) \left|
  \begin{array}{l}
  X \in \Sha_M^{I \hbox{-}\hom_\ell}(F,R) \text{ with } \Aut_{I}(X)=\{1\}  \\
  \lambda \colon \text{polarization of BB degree} \ d
  \end{array}
  \right.
\right\}/{\cong_F}.
\end{equation}
Note that we may assume $1/\ell \in R$ and $R$ is regular.

In the following, we fix an integer $m$ which is a sufficiently large power of $\ell$. To avoid using the unramifiedness of $2$-adic cohomology groups, we consider the congruence level-$m$ subgroup
\[
\widetilde{\mathbf{K}}_m \coloneqq \left\{g \in G_\Sigma(\widehat{\ZZ}) \big| g \equiv 1~\bmod m \right\}
\]
instead of the spin-level-$m$ subgroup  in \eqref{eq:bigdiag}.
We have the following diagram:
\[
\begin{tikzcd}[column sep = small]
      & & & \Sh_{K^{\operatorname{sp}}_{\Sigma,m}}(\GSpin(\Sigma), \widetilde{D}_{\Sigma}) \ar[d,"\spin"] \ar[ld, shift left,"f_\Sigma"'] \\
\widetilde{\scF}^{\dag}_{d,K_{h,m}} \ar[r,"\mathcal{P}"] & \Sh_{K_{h,m}}(G_{\Lambda_h},D_{\Lambda_h}) \ar[r,"j_{\Sigma}"] & \Sh_{\widetilde{\mathbf{K}}_m}(G_\Sigma, D_{\Sigma})   & \Sh_{\mathbf{K}''_m}(\GSp(W), \Omega^{\pm}).
\end{tikzcd}
\]
We shall remark that the construction of a section for $f_\Sigma$ as in \eqref{eq:bigdiag}  fails as $\mathbf{K}'_m \cap G_{\Sigma}(\QQ) \subset \widetilde{\mathbf{K}}_m \cap G_{\Sigma}(\QQ)$ may be a proper arithmetic subgroup.

Consider the following set 
\[
S_{2}\coloneqq
\left\{
  (X, \lambda, \omega, \alpha) \left|
  \begin{array}{l}
(X, \lambda, \omega, \alpha) \in  \widetilde{\scF}^{\dag}_{d,K_{h,m}}(F) \\
 \rH^{2}_{\et}(X_{\overline{F}},\ZZ_{\ell}) \colon 
  \text{unramified at any height 1 prime } \p\in \Spec R 
  \end{array}
  \right.
\right\}/{\cong_F}
\]
Firstly, we show that there are only finitely many geometric isomorphism classes in $S_2$.
By \cite[Section 2, Section 3]{Kisin-IntegralModel}, the map $f_{\Sigma}$ can be extended to a morphism of integral models over $R$
\[
\widetilde{f}_{\Sigma}\colon \scS_{K^{\Sp}_{\Sigma,m}}(\GSpin(\Sigma)) \rightarrow \scS_{\widetilde{\mathbf{K}}_m}(\Sigma).
\]
We may assume that $\widetilde{f}_{\Sigma}$ is a finite \'{e}tale cover.
A similar proof as in Lemma \ref{unplevel} shows that there exists a finite extension $F_{M}/F$ such that, for any object $x \in S_{2}$, the $F$-rational point $j_{\Sigma} \circ \mathcal{P}(x)$ can be lifted to an $F_{M}$-valued point of $z_{x}$ of $\Sh_{K^{\Sp}_{\Sigma,m}}(\GSpin(\Sigma))$.
Combining \cite[VI, \S1, Theorem 2]{Faltings1992}, Lemma \ref{lemKSgood} and the quasi-finiteness of $j_{\Sigma} \circ \mathcal{P}$, we can see the set of geometrical isomorphism classes in $S_2$ is finite.

By the choice of the level structure, the unramifiedness condition, and the proof of Lemma \ref{unplevel}, one can show there exists a finite Galois extension $E$ over $F$ such that for every $(X,\lambda) \in S_{2}$, there exists a level $\widetilde{\mathbf{K}}_m$-structure on $(X,\lambda)_{E}$. Now, by the same argument as in the proof of Lemma \ref{unplevel} and Theorem \ref{thmpol}, the problem is reduced to show that
\[
S_3\coloneqq
\Set*{
  (X, \lambda, \omega, \alpha,\beta) \given 
\parbox{16em}{
$(X, \lambda, \omega, \alpha) \in S_2$ \\
$\beta$: level $\widetilde{\mathbf{K}}_m$-structure on  $X$  of degree  $I$ \\
$\Aut_{I}(X)=\{1\}$}
  }/{\cong_F}
\]
is a finite set.
Since the automorphism group of any object $(X, \lambda, \omega, \alpha,\beta) \in S_3$ is trivial, the finiteness of $S_3$ follows from the finiteness of $S_{2}$ modulo geometric isomorphic equivalences. It finishes the proof of the polarized case.

Finally, for the general (unpolarized) case, by the choice of $\widetilde{\mathbf{K}}_m$ and Lemma \ref{unplevel}, we can take a finite extension $E/F$ which satisfies the following:
\begin{enumerate}
\item
$E$ contains the above field $F_{M}$.
\item
For any element $X \in \Sha_M^{I\hbox{-}\hom_\ell}(F,R)$, there exist a polarization $\lambda$ on $X$, an orientation $\omega$ on $X_{E}$, and a level $\widetilde{\mathbf{K}}_m$ structure $\alpha$ on $X_{E}$.
 \item
 For any element $X \in \Sha_M^{I \hbox{-}\hom_\ell}(F,R)$, we have $\Pic_{X/F}(E) = \Pic_{X/F} (\overline{F})$.
\end{enumerate}
By a similar proof of Theorem \ref{boundedpolarization}, one can also show that
the set
\[
\{
\Pic_{X/F}(E) \mid X \in \Sha_M^{I \hbox{-}\hom_\ell}(F,R)
\}/ \text{lattice isometry}
\]
is a finite set.
Using the same argument in Theorem \ref{finbirsh} and \cite[Theorem 1.0.1]{takamatsu21}, we can get the boundedness of polarization when $b_2 \geq 5$. The desired finiteness follows from the finiteness of $S_{1}$.
\end{proof}

\subsection{Some consequences}
We deduce from Theorem \ref{cohshaf} finiteness results of the (second or full) cohomological Shafarevich sets for projective hyper-Kähler varieties of known deformation types and unconditionally in dimension 4:
\begin{corollary}
\label{cor:cohshaf}
The  following sets are finite.
\begin{enumerate}
    \item $\Sha_{K3^{[n]}}^{\hom_\ell}(F,R)$ and $\Sha_{OG_{10}}^{\hom_\ell}(F,R)$.
    \item  $\Sha_{\operatorname{Kum}_n}^{I \hbox{-}\hom_\ell}(F,R)$ and $\Sha_{OG_{6}}^{I \hbox{-}\hom_\ell}(F,R)$, for $I=\NN$.
    \item $\Sha_{M}^{I-\hom_\ell}(F,R)$ for $\dim(M)=4$ and $I=\mathbb{N}$.
\end{enumerate}
\end{corollary}

\begin{proof}
By Theorem \ref{cohshaf}, it suffices to check the vanishing of $\Aut_0$ or $\Aut_{00}$ for the corresponding complex projective hyper-Kähler varieties. These vanishing results are collected in Example \ref{exm:aut}.
\end{proof}
\begin{remark}
In Corollary \ref{cor:cohshaf}, for the generalized Kummer type and $OG_6$ type, one may be able to choose $I$ to be a smaller set. For example, for the $OG_6$-type, $I=\{2, 6\}$ is in fact enough: for some projective $OG_6$-type hyper-Kähler manifolds $X$ constructed as crepant resolutions of singular moduli spaces $\overline{X}$, the 3-dimensional subvarieties lying over the 256 singular points of the singular locus of $\overline{X}$ are with distinct cohomology classes in $\rH^6(X,\QQ)$ and are permuted by $\Aut_0(X)$ as a regular representation. Then by Proposition \ref{prop:Aut00} (which holds more generally for $\Aut_I$ with the same proof) and the density of projective members in the moduli space, we see that $\Aut_{\{2, 6\}}$ vanishes for all $OG_6$-type manifolds.
\end{remark}

\section{Finiteness of CM type hyper-Kähler varieties}  
\label{sec:FiniteCM}
Recall that a hyper-Kähler variety $X$ over $\CC$ is called of \emph{CM type} if the Mumford--Tate group $\mathrm{MT}(\rT(X))$ is abelian, where $\rT(X)\subset \rH^2(X, \QQ)$ is the transcendental (rational) cohomology. The following results allow us to study hyper-Kähler varieties of CM type through abelian varieties of CM type. 
\begin{lemma}\label{KS-CM}
 The Kuga--Satake variety of a complex hyper-Kähler variety $X$ is of CM type if and only if $X$ is of CM type. In particular, the point
 \[
 \Psi^{\KS}_{d,\mathbf{K}_4}(X,\xi,\omega; \alpha) \in \Sh_{\mathbf{K}''_4}(\GSp(W), \Omega^{\pm})(\CC)
 \]
is a CM point if $X$ is of CM type. 
\end{lemma}

\begin{proof}
A proof for the statement for K3 surfaces can be found in, for example, \cite[Proposition 9.3]{CMlifting}, which works for general Hodge structures of K3 type. We sketch it here for the convenience of the readers. 

Let $A$ be the (uniform) Kuga--Satake abelian variety of $X$. Let $\MT(A)$ be the Mumford--Tate group of $W \coloneqq \rH_1(A, \QQ)$. The Hodge structure of $W$ factors through $\widetilde{h}$ in the commutative diagram
\[
\begin{tikzcd}
 \mathbb{S} \ar[r,"\widetilde{h}"] \ar[rd,"h"']&  \GSpin(\Sigma)_{\mathbb{R}} \ar[d,"\ad"] \ar[r,hook] & \GSp(W) \\
 & G_{\Sigma,\mathbb{R}} &
\end{tikzcd}
\]
Therefore, the Mumford--Tate group $\MT(A)$ is contained in $\GSpin(\Sigma)_\QQ$. The commutativity of the diagram implies that $\MT(\rT(X))$ is contained in $\ad(\MT(A))$.

If $A$ is of CM type, then $\MT(A)$ is abelian, and so is its image in $G_{\Sigma}$ under $\ad$. Therefore, $\MT(\rT(X))$ is also abelian. Conversely, we note that $\widetilde{h}(\mathbb{S}) \subseteq \ad^{-1}(\MT(\rT(X))_{\mathbb{R}}$. Thus $\MT(A)$ is a subgroup of $\ad^{-1}(\MT(\rT(X))$, which are both solvable if $\MT(\rT(X))$ is abelian. Since $\MT(A)$ is reductive, $\MT(A)$ is abelian if $\MT(\rT(X))$ is abelian.
\end{proof}
\begin{remark}
The lemma \ref{KS-CM} also implies that for a hyper-Kähler variety $X$ over a field $F$ of characteristic zero, being of CM type is independent of the embedding $F\hookrightarrow \CC$.

\end{remark}
The following finiteness result for abelian varieties of CM type is obtained by Orr--Skorobogatov.  
\begin{theorem}[{\cite[Theorem 2.5]{OrrSkorobogatov}}]\label{CMab-fin}
The geometric isomorphism classes of abelian varieties of CM type defined over a number field of bounded degree form a finite set. 
\end{theorem}
Moreover, they used this to deduce a finiteness result of CM points on a Shimura variety of abelian type.
\begin{theorem}[{\cite[Proposition 3.1]{OrrSkorobogatov}}]\label{CM-abfin}
Let $\Sh$ be a component of the complex points of a Shimura variety of abelian type. The set of CM points in $\Sh$ defined over a number field of bounded degree is finite.
\end{theorem}

We now give a proof of the generalizations of these results for hyper-Kähler varieties of CM type.

\begin{proof}[Proof of Theorem \ref{mainthm3}]
We need to show the finiteness of the following set:
\begin{equation*}
\mathtt{CM}_{d}(M)=\bigg\{ X\bigg|\begin{array}{lr}
	\hbox{\rm  $X$~is a hyper-Kähler variety over a number field $F$ of  degree $\leq d$  } &  \\
	\hbox{\rm   $X_\CC$ is of deformation type $M$ and of CM type for some $F\hookrightarrow \CC$} &
	\end{array}  \bigg \}.
\end{equation*}
We first prove that the isometry class of  $\{ \NS(X_\CC) |~X\in \mathtt{CM}_{d}(M)\}$ is finite. By Proposition \ref{unplevel} (i), for any positive integer $N$, we can find an integer $\delta(4,M)$ such that for every $X/F \in \mathtt{CM}_d(M)$ with a polarization $H$ of degree $\leq N$, it has a $\mathbf{K}_4$-level structure  over a finite extension $F'/F$ with $[F':F]\leq \delta(4,M)$. Note that $\delta(4,M)$
depends only on the $\widehat{\ZZ}$-BB form and hence is independent of the field $F$ and the embedding $F\hookrightarrow \CC$.

By Lemma \ref{KS-CM} and  Theorem \ref{CM-abfin},  the collection 
\begin{equation}
    \bigg\{ \Psi^{\KS}_{d,\mathbf{K}_4}(X_\CC, \xi_\CC,\omega_{\CC};\alpha_\CC)\bigg|\begin{array}{lr}
	\hbox{\rm  $(X,\xi,\omega;\alpha)$ is a polarized oriented hyper-Kähler with a $\mathbf{K}_{4}$-level  } &  \\
	\hbox{\rm  structure defined over $F'$  with  $[F':\QQ]\leq \delta(4,M)d$; $X\in \mathtt{CM}_{d}(M)$  } &
	\end{array}  \bigg \}
\end{equation}
is finite. By a similar argument in Theorem \ref{boundedpolarization}, one can obtain the finiteness of the set of isometry classes of geometric transcendental lattices $\rT(X_\CC)$, which implies the finiteness of the set of isometry classes of $\Pic(X_\CC)$. 

According to the proof in Theorem \ref{finbirsh}, up to a birational transformation, there exists a polarization on $X_\CC$ of degree $\leq N$ for some $N$.    Note that the geometric isomorphism classes of the polarized hyper-Kähler varieties $(X,\xi)$ with $X\in \mathtt{CM}_d(M)$ and $(\xi)^2\leq N$ are finite because the Kuga--Satake map over $\CC$ is quasi-finite.  Thus, the geometric birational isomorphism class of $\mathtt{CM}_d(M)$ is finite. 

When $b_2\geq 5$, the last assertion follows from Theorem \ref{cor:bir-fin}. 
\end{proof}

\begin{proof}[Proof of Corollary \ref{CorBr}]

For a hyper-Kähler variety $Y$, the uniform boundedness of \[ \lvert \Br(Y)/\Br_0(Y) \vert\] follows from the uniform boundedness of \[\lvert \Br(Y)/\Br_1(Y) \rvert,\]
by a similar proof to \cite[Proposition 6.3]{VV17}. Here, $\Br_1(Y) \coloneqq \ker(\Br(Y) \to \Br(Y_{\bar{F}})$ is the \emph{algebraic Brauer group} of $Y$. On the other hand, we have an inclusion
\[
\Br(Y)/\Br_1(Y) \subseteq \Br(Y_{\bar{F}})^{\Gal(\bar{F}/F)}.
\]
Therefore, it is sufficient to prove the uniform boundedness of $\lvert \Br(Y_{\bar{F}})^{\Gal(\bar{F}/F)} \rvert$.

According to \cite[Theorem 5.1]{OrrSkorobogatov},  for any positive integer $d$, there is a constant $C(d,X)$ such that, for any $L$-form $Y$ of $X$ with $[L:F] \leq d$, 
\[ 
\lvert \Br(Y_{\bar{F}})^{\Gal(\bar{F}/L)} \rvert < C(d,X).
\]
if the integral Mumford--Tate conjecture in codimension one holds. By \Cref{mainthm3}, there is a constant $C(d)$ which is independent of $X$ and $C(d,X) \leq  C(d)$. Then we can conclude the uniform boundedness.

We claim that the integral Mumford--Tate conjecture holds true for hyper-Kähler varieties with $b_2 \geq 4$. From \cite[Proposition 6.2]{CadoretMoonen}, we can see that the Hodge structure on $\rH^2(X,\ZZ)$ is Hodge-maximal. Then we can follow the proof of Theorem 6.6 in {\it loc.cit.} to show that the classical Mumford--Tate conjecture of $X$ implies the integral version, in which we need the arithmetic period map given in Proposition \ref{prop:descentperiod}. However, the Mumford--Tate conjecture in codimension one holds true for all hyper-Kähler varieties of $b_2 \geq 4$ (see \cite[Corollary 1.5.2]{A96}).
\end{proof}

\appendix

\section{Matsusaka--Mumford theorem for algebraic spaces}
\label{appendixMatsusakaMumford}
The Matsusaka--Mumford theorem \cite{MM64}
states that a birational map between two algebraic varieties
over the fraction field of a discrete valuation ring
specializes to a birational map between the special fibers
if the special fibers are non-ruled.
In this appendix, we extend this theorem to
birational maps between algebraic spaces that are not necessarily schemes and study the specialization homomorphism. 
(Such birational maps appear naturally in the study of
hyper-Kähler varieties. See \cite[Section 4.4]{LM18}.)
The result should be known to experts, but we could not find an appropriate reference. 

We recall that for an integral and separated algebraic space $X$, there exists an open subspace that is a scheme and contains a generic point of $X$. Moreover, there exists a largest such open subspace, called the \textit{schematic locus} of $X$.
\begin{definition}
\label{def:Ruled}
Let $X$ be an $n$-dimensional integral separated algebraic space of finite type over a field $k$.
We say $X$ is ruled if the schematic locus $X'$ of $X$ is ruled, i.e.,$X'$ is $k$-birationally isomorphic to $\PP^{1}_{k} \times Y$ for some scheme $Y$ of dimension $n-1$ over $k$. 
\end{definition}

In the sequel,  let $R$ be a discrete valuation ring, $K = \Frac R$ its fraction field, $s$ the closed point of $\Spec R$, and $k(s)$ the residue field. 

\begin{theorem}[Matsusaka--Mumford for algebraic spaces]
\label{ClosedSubset}
Let $\mathcal{X}, \mathcal{Y}$ be integral smooth proper algebraic spaces over $R$.
Suppose that the special fiber $\mathcal{Y}_{s}$ is not ruled.
Let $f \colon \mathcal{X}_{\eta} \overset{\simeq}{\dashrightarrow} \mathcal{Y}_{\eta}$
be a birational map between the generic fibers.
Then there exist closed algebraic subspaces
$V \subset \mathcal{X}$, $W \subset \mathcal{Y}$
with special fibers $V_{s}, W_{s}$ satisfying $V_{s} \neq \mathcal{X}_{s}$, $W_{s} \neq \mathcal{Y}_{s}$,
such that
$f$ extends to an isomorphism over $R$
\[
  \widetilde{f} \colon 
  \mathcal{X} \setminus V \overset{\simeq}{\to}
  \mathcal{Y} \setminus W.
\]
\end{theorem}
\begin{proof}
The same proof as in \cite[Theorem 1]{MM64} (see also \cite[Theorem 5.4]{Matsumoto2015}) works, by using the notion of generic points and Henselian local rings of decent algebraic spaces.
\end{proof}

We thank Tetsushi Ito for his help with the proof of the following lemma.

\begin{lemma}
\label{Lemma:Specialization}
Let $\mathcal{X}$
be an integral proper smooth algebraic space over $R$.
Suppose that the special fiber $\mathcal{X}_{s}$ is not ruled.
Let $\Bir(\mathcal{X}_{\eta})$ be the group of $F$-birational automorphisms
of the generic fiber $\mathcal{X}_{\eta}$.
Let $G$ be a finite subgroup of $\Bir(\mathcal{X}_{\eta})$.
\begin{enumerate}
\item There exists a proper closed subspaces $V \subset \mathcal{X}$
with $V_s \neq \mathcal{X}_{s}$ and a homomorphism
\[
\psi \colon G \to \Aut_R(\mathcal{X} \setminus V)
\]
such that $\psi(f)|_{\mathcal{X}_{\eta}} = f$ for any $f\in G$.

\item Let $\overline{\psi}$ be the specialization morphism defined as the composition 
\[
G \overset{\psi}{\to} \Aut_R(\mathcal{X} \setminus V)
 \to \Aut((\mathcal{X} \setminus V)_{s})
\]
given by $\overline{\psi}(f) \coloneqq \psi(f)|_{(\mathcal{X} \setminus V)_{s}}$.
If the characteristic of the residue field of $R$ is $0$,
the map $\overline{\psi}$ is injective.
If the characteristic of the residue field of $R$ is $p > 0$,
the kernel of $\overline{\psi}$ is a $p$-group.
\end{enumerate}
\end{lemma}

\begin{proof}
(1) We put $G = \{ f_1=\mathrm{id},\ldots,f_r \}$.
Let $V_{f_i}, W_{f_{i}} \subset \mathcal{X}$ be proper closed subspaces
associated with $f_i$ given by Theorem \ref{ClosedSubset} (i.e.\,$f_{i}$ extends to an isomorphism $\widetilde{f}_{i}\colon \mathcal{X} \setminus V_{f_{i}} \simeq \mathcal{X} \setminus W_{f_{i}}$).
Here, we put $V_{f_{1}} = \emptyset$.
We put $V'\coloneqq\bigcup_{i = 1}^{r} V_{f_i} $,
and 
\[
V\coloneqq \bigcup_{i=1}^{r} \widetilde{f}_{i}^{-1}(V').
\] 
First, since $|\widetilde{f}_{i}^{-1}(V') \cup V_{f_{i}}|$ is a closed subset of $|\mathcal{X}|$, the subset $V$ is a closed subspace of $\mathcal{X}$.
Moreover, since we have 
\[
|\widetilde{f}_{i}^{-1}( \widetilde{f}_{j}^{-1}(V'))| \subset
|\widetilde{(f_{j}\circ f_{i})}^{-1}(V') \cup V_{f_{j} \circ f_{i}}|,
\]
we have $\widetilde{f}_{i}^{-1}(V) \subset V$.
Since $V' \subset V$, the morphism $\widetilde{f}_{i}$ restricts to an  $R$-morphism 
\[
f_{i}' \colon
\mathcal{X} \setminus V \rightarrow \mathcal{X} \setminus V.
\]
Since $(f_{i})' \circ (f_{i}^{-1})' =(f_{i}^{-1})' \circ (f_{i})' = \mathrm{id}$,
$f_{i}'$ is an isomorphism, and $V$ satisfies the desired condition.

(2)
Take an automorphism $f \in \ker \overline{\psi}$.
Let $\ord(f)$ be the order of $f$.
Assume that $\ord(f)$ is invertible in the residue field of $R$.
We shall show that $f$ is the identity.

Since $f$ acts trivially on $(\mathcal{X} \setminus V)_{s}$,
it fixes every closed point $x \in (\mathcal{X} \setminus V)_{s}$.
Let $k(s)$ be the residue field of $R$.
We can take a $\overline{k(s)}$-valued point $x$ of $(\mathcal{X} \setminus V)_{s}$, such that $x$ is contained in the schematic locus of $(\mathcal{X} \setminus V)_{s}$.
It is enough to show that the action of $f$ on the local ring $\mathscr{O}_{\mathcal{X},x}$ is trivial. 
Let $\pi$ be a uniformizer of $R$.
Since $\pi$ is not invertible in $\mathscr{O}_{\mathcal{X}, x}$,
we have $\bigcap_{n \geq 1} \pi^n \mathscr{O}_{\mathcal{X}, x} = (0)$ by Krull's intersection theorem; see \cite[Theorem 8.10]{Matsumura1989}.
Therefore, it is enough to show that
the action of $f$ on 
$\mathscr{O}_{\mathcal{X}, x}/\pi^n \mathscr{O}_{\mathcal{X}, x}$
is trivial for every $n \geq 1$.
We shall prove the assertion by induction on $n$.
We put $A_n \coloneqq \mathscr{O}_{\mathcal{X}, x}/\pi^{n} \mathscr{O}_{\mathcal{X}, x}$.
When $n = 1$, it follows because $f$ acts trivially on
the special fiber $(\mathcal{X} \setminus V)_{s}$.
If the assertion holds for some $n \geq 1$, the exact sequence
\[
0 \to \pi^n A_{n+1}
  \to A_{n+1}
  \to A_{n}
  \to 0
\]
shows that, for an element $a \in A_{n+1}$,
we have $f^{\ast}(a) = a + \pi^n b$ for some $b \in A_{n+1}$
(here the action of $f$ on $A_{n+1}$ is denoted by $f^{\ast}$).
Since $\pi^{n+1} = 0$ in $A_{n+1}$,
we have $(f^{\ast})^r(a) = a + r \pi^n b$ for every $r \geq 1$.
In particular, we have
\[ (f^{\ast})^{\ord(f)}(a) = a + \ord(f) \pi^n b = a. \]
Since $\pi^n A_{n+1}$ is a vector space over $k(x)$
and $\ord(f)$ is invertible in the residue field,
we have $\pi^n b = 0$.
Thus, the action of $f$ on $A_{n+1}$ is trivial.
By induction on $n$, the assertion is proved.
\end{proof}

\begin{remark}
If $\Aut(\mathcal{X}_{\eta})$ is infinite,
a homomorphism
\[
\Aut(\mathcal{X}_{\eta}) \to \Aut(\mathcal{X} \setminus V)
\]
may not exist for any choice of $V \subset \mathcal{X}_s$.
The problem is that there might be infinitely many closed subsets
that should be removed from $\mathcal{X}$
if we apply Lemma  \ref{ClosedSubset}.
On the other hand, there is a well-defined specialization homomorphism
\[ \Bir(\mathcal{X}_{\eta}) \to \Bir(\mathcal{X}_s), \]
where $\Bir(-)$
denotes the group of birational automorphisms;
see \cite[p.\ 191-192, Exercises 1.17]{Kollar-RationalCurvesBook}.
\end{remark}

\begin{remark}
There does not exist a specialization map
\[ \Aut(\mathcal{X}_{\eta}) \to \Aut(\mathcal{X}_s) \]
in general.
The problem is that an automorphism of the generic fiber
does not always extend to an automorphism of the integral model.
Such an extension exists for abelian varieties (\cite[Theorem 1.4.3]{BLR90}), but fails in general for 
for K3 surfaces (\cite[Example 5.4]{LM18}). 
However, since any birational map between K3 surfaces can be extended to an automorphism, there still exists a map $\Aut(\mathcal{X}_{\eta}) \to \Aut(\mathcal{X}_s)$ for K3 surfaces.
On the other hand, suppose that there exists a polarization on $\mathcal{X}$ (this does not hold in general, see \cite[Example 5.2]{Matsumoto2015}), then by Matsusaka--Mumford \cite[Theorem 2]{MM64},
there always exists a specialization map
\[
\Aut(\mathcal{X}_{\eta}, \mathcal{L}_{\mathcal{X}_{\eta}}) \to
\Aut(\mathcal{X}_{s}, \mathcal{L}_{\mathcal{X}_{s}}).
\]
provided that the special fiber is non-ruled.
This map is used in Andr\'e's proof of the Shafarevich
conjecture for (very) polarized hyper-Kähler varieties;
see the proof of \cite[Lemma 9.3.1]{A96}.
\end{remark}

\bibliographystyle{plain}
\bibliography{SHK}
\end{document}